\documentclass[12pt]{article}

\usepackage[a4paper,margin=1in]{geometry}

\usepackage{amsmath,amssymb,amsfonts,amsthm}
\usepackage{mathrsfs}
\usepackage{hyperref}
\usepackage{enumitem}
\usepackage{orcidlink}
\usepackage{verbatim}
\usepackage{tikz}
\usepackage{xcolor}
\usepackage{comment}

\numberwithin{equation}{section}

\theoremstyle{plain}
\newtheorem{theorem}{Theorem}[section]
\newtheorem{lemma}[theorem]{Lemma}
\newtheorem{corollary}[theorem]{Corollary}
\newtheorem{proposition}[theorem]{Proposition}

\theoremstyle{definition}
\newtheorem{definition}[theorem]{Definition}
\newtheorem{example}[theorem]{Example}

\theoremstyle{remark}
\newtheorem{remark}[theorem]{Remark}

\newcommand{\trygve}[1]{\textcolor{blue}{#1}}

\newcommand{\F}{\mathbb{F}}

\newcommand{\C}{\mathbb{C}}

\newcommand{\OS}{\text{OS}}
\newcommand{\Ecal}{\mathcal{E}}
\newcommand{\Ical}{\mathcal{I}}

\newcommand{\rank}{\operatorname{r}}

\title{Generalized Weight Polynomials of Codes through Flats and  Orlik-Solomon Algebras of Matroids}

\author{%
  Jakob Anderssen\,\orcidlink{0000-0002-6982-7079}\\
  \small Department of Applied Mathematics and Computer Science\\
  \small Technical University of Denmark, Lyngby, Denmark\\
  \small\texttt{vpnet11@gmail.com}
  \and
  Trygve Johnsen\\
  \small Department of Mathematics and Statistics\\
  \small UiT The Arctic University of Norway, Tromsø, Norway\\
  \small\texttt{Trygve.Johnsen@uit.no}
}

\date{} 

\begin{document}

\maketitle

\let\thefootnote\relax\footnotetext{

Trygve Johnsen was partially supported by the project Pure Mathematics in Norway, funded by Bergen Research Foundation and Troms\o\ Research Foundation and the UiT-project Mascot. }


\begin{abstract}
We present various ways of determining generalized weight polynomials of a matroid $M$, and we recall how one can find the generalized weight spectra of a linear code, given these polynomials, for the matroid determined by any generator matrix of the code. A main goal is to give coding theorists different ways to determine these polynomials.

We describe how one can find the generalized weight polynomials of any matroid $M$, directly from its lattice of flats, and we also show how one can find them from the Poincare series (in this case  polynomials) of the associated Orlik-Solomon algebras of matroids arising as  contractions of the flats of $M$.
This opens for using information about broken circuits of the matroids to determine generalized weight polynomials.
We also recall the well-known connection between the Orlik-Solomon algebra of a matroid, and Whitney numbers obtained by order homology, and use it to show how one can describe weight polynomials in terms of Whitney numbers from order homology of the matroid and its contraction of flats. 

We recall briefly a well-known relation between the Orlik-Solomon algebra of a matroid $M$, which is representable over the complex numbers,  and de Rham homology numbers obtained from complements of intersections of hyperplanes in a hyperplane arrangement corresponding to $M$. We also describe a way to find the generalized weight polynomials of a matroid in terms of polynomials defined in connection with its lattice of cyclic flats.

In a simple running example we show how one can calculate the generalized weight polynomials in different ways, including both the methods presented in this paper, and selected methods developed in earlier papers. We also include a less simple example with the projective Reed-Muller code $PR_3(2,2).$

\end{abstract}

\section{Introduction}

When working with linear error-correcting codes $C$ of length $n$ over a finite field $\mathbf{F}_q$, an interesting issue is how many codewords there are of each of the weights $w=0.1,\cdots,n$. This is called finding the weight distribution of the code. A more ambitious goal is to determine the number $A_w^{(r)}$ of subcodes of dimension $r$ and support weight $w$, for each possible $w$ and $r$. This is called finding the higher weight spectra, and finding these spectra clearly also enables us to find the generalized Hamming weights and the minimum distance of the code.
As a tool for determining these spectra one may first study the extended codes 
$C_m=C\otimes_{\mathbf{F}_q}\mathbf{F}_{q^m}$ 
and finding the  weight distribution of all these codes.

It turns out that the number of codewords of $C_m$ of weight $w$ is
$P_w(q^m)$, where $P_w(Z)$ is a polynomial in degree at most the dimension of the code, and that determining these finitely many polynomials of bounded degree thus gives the weight distributions for all the infinitely many extension codes simultaneously. The $P_w(C)$ are called the generalized weight polynomials of the code $C$. The conversion formula 
\begin{equation} \label{conversion}
P_w(q^m)=\sum_{r=0}^mA_w^{(r)}\prod_{i=0}^{r-1}(q^m-q^i), \textrm{for }m>0,
\end{equation}
given for example in \cite{HellesethKloveMykkeltveit} or \cite[Proposition 6]{JurriusPellikaanCodes}, with the coefficients of the $P_w(Z)$ as input, can then be applied to find the higher weight spectra $A_w^{(r)}$ as output (or vice versa if one happens to find the $A_w^{(r)}$ first).

In this note we will recall the description of  the generalized weight polynomials in a more generalized setting, for matroids. This description "matches" the description for linear codes, when the matroid comes from a linear code (from a generator matrix, or from a parity check matrix, according to convention).
We will recall some well-known results
and relate the weight polynomials to characteristic polynomials, chromatic polynomials, Orlik-Solomon algebras, and Whitney numbers, and consider connections with other concepts. The primary goal of this article  is  to view generalized   weight polynomials from different angles and thus obtain a broad matroid-based understanding of these objects. Furthermore, as a consequence of this,  we hope to present useful tools for code-theorists who want to find generalized weight polynomials and weight spectra of concrete linear codes.

Greene's work~\cite{Greene1976} showed that ordinary weight enumerators can be expressed in terms of the Tutte polynomial of an associated matroid. It is even true that all the generalized weight polynomials can be found if one knows the Tutte polynomial. The Tutte polynomial depends on how many of \underline{all} the subsets of the ground set
 that have given  (matroid) ranks, for all possible ranks. So finding this information for \underline{all} the subsets is clearly one direction, in which one can find the generalized weight polynomials. 
 Another direction, which is one that will be in focus here, is to find the generalized weight polynomials, only through extracting  information about the flats of the matroid from a generator matrix of the linear code, or equivalently, the cycles of the matroid from a parity check  matrix of the code. The inclusion relations and cardinalities of those flats will determine the generalized weight polynomials, and we will show in what way.
For linear codes coming from incidence matrices of graphs we will point out how the generalized weight polynomials can be read off from chromatic polynomials of the graph and its contractions at relevant subgraphs.


Moreover Orlik and Solomon~\cite{OrlikSolomon,OST1984} defined a graded algebra $\OS(M)$ attached to a matroid $M$, now known as the Orlik--Solomon (OS) algebra, which encodes the combinatorics of circuits and broken circuits. This non-commutative algebra is isomorphic to the Whitney  cohomology ring coming from order complexes of all the flats of the matroid. In the cases when the matroid is representable over the complex numbers, it is also isomorphic  to the cohomology ring of complements of complex hyperplane arrangements representing $M$.
These  classical observations in  \cite{OrlikSolomon,OST1984} and  \cite{Bjorner}, enable us to combine the information we have about how to derive the generalized weight polynomials from the information about the flats of our matroid, with information about the Poincare series (which is a polynomial) of the OS-algebra.
This gives a new computational tool to find the generalized weight polynomials, by only studying broken circuits of the matroid and its contractions at flats.

To give a broader picture we  will also show how one can give formulas for generalized weight polynomials for a matroids, in terms of its lattice of cyclic flats, and associated functions.

In a running example with a graphical matroid, we show how one can calculate the generalized weight polynomials in different ways, including both the methods presented in this paper, and selected methods developed in earlier papers. We also include a less simple example with the projective Reed-Muller code $PR_3(2,2).$

\section{Preliminaries}

\begin{definition}
A simplicial complex $\Delta \subseteq 2^V$ is downward closed: if $F \in \Delta$ and $G \subseteq F$, then $G \in \Delta$. Faces are elements of $\Delta$; maximal faces are called facets.
\end{definition}

\begin{definition}
For a face $F$, define $\dim(F)=|F|-1$. The dimension of $\Delta$ is
\begin{equation*}
\dim(\Delta)=\max_{F \in \Delta}\dim(F).
\end{equation*}
The complex $\Delta$ is \emph{pure} if all facets have equal size.
\end{definition}

\begin{definition}
Let $r=\dim(\Delta)+1$. The face enumerator is
\begin{equation*}
f(\lambda)=\sum_{i=0}^{r} f_i \lambda^{r-i},
\quad
f_i = \#\{F \in \Delta : |F|=i\}.
\end{equation*}
\end{definition}

\begin{definition}
The (reduced) Euler characteristic of $\Delta$ is
\begin{equation*}
\chi(\Delta)=(-1)^{r-1}f(-1).
\end{equation*}
\end{definition}

\begin{proposition}
If $\Delta$ is a cone, then $\chi(\Delta)=0$.
\end{proposition}

\begin{proof}
Faces split into disjoint pairs $\{F, F \cup \{v\}\}$, which cancel in the alternating sum defining the Euler characteristic.
\end{proof}

\begin{definition}
A pure simplicial complex $\Delta$ is shellable if its facets $F_1,\dots,F_t$ can be ordered such that for each $i = 2,\dots,t$,
\begin{equation*}
\left( \bigcup_{j<i} F_j \right) \cap F_i
\end{equation*}
is a pure simplicial complex of dimension $\dim(F_i)-1$.

Equivalently, for all $i = 2,\dots,t$, there exists $k<i$ such that
\begin{equation*}
F_k \cap F_i \subseteq F_j \cap F_i \quad \text{for all } j<i,
\quad \text{and} \quad |F_k \cap F_i| = |F_i|-1.
\end{equation*}
\end{definition}






\begin{definition}\label{def:matroid} A (finite) matroid is a pair $(E,r)$, where $E=\{1,\cdots,n\},$ and $r$ is a rank function 
$r:2^E\rightarrow \mathbb{N}_0$ satisfying \begin{description}
\item[(R1)] If $X \subset E$, then $0 \le r(X) \le |X|.$
\item[(R2)] If $X \subset Y \subset E$, then $r(X) \le r(Y).$
\item[(R3)] If $X$ and $Y$ are subsets of $E$, then $r(X \cup Y)+r(X\cap Y) \le r(X)+r(Y).$
\end{description}
We denote by $r(M)$ the integer $r(E).$
\end{definition}

\begin{definition} \label{dual}
The dual matroid $M^*$ for a matroid $M$ is the pair $(E,r^*),$ where $r^*(X)=|X|+r(E\setminus X)-r(E),$ for all $X \subset E.$
\end{definition}

\begin{definition}
For a matroid the subsets $A$ of $E$, such that $r(A)=|A|$ are called independent. The remaining subsets are called dependent.
We define
\begin{equation*}
IN(M) = \{I \subseteq S : I \text{ is independent}\}.
\end{equation*}
The bases of $M$ are the facets of $IN(M).$
\end{definition}

\begin{theorem}\label{thm:shellability}
\leavevmode \\
The independence complex $IN(M)$ is a pure and shellable simplicial complex.
\end{theorem}

\begin{proof}
Purity follows since all bases have equal size (rank). Any lexicographic ordering of bases under a total ordering of elements satisfies the exchange-based shelling criterion; see e.g.~\cite{Bjorner,Oxley1992}.
\end{proof}

\begin{definition} \label{basic}
Let $M=(E=E(M),r=r(M))$ be a matroid, and let $ X \subset E.$ 
\begin{enumerate}
\item The nullity of $X$ is $n(X)=|X|-r(X)$. We denote by $n^*(X)=|X|-r^*(X)$ the nullity of $X$ for the dual matroid $M^*=(E,r^*).$
\item $X$ is called a cycle of $M$ if $n(Y]<n(X)$ for all strict subsets $Y \subset X.$
\item A circuit is a cycle of nullity $1$ (inclusion minimal dependent set). 
\item A flat of $M$ is a subset $X \subset E$ such that $r(X) < r(Y)$ for all subsets $Y \subset E$ such that $X$ is strictly contained in $Y.$
\item A loop of $M$ is an element $ x \in E$ such that $r(\{x\})=0.$ A coloop is a loop of $M^*.$
\item A matroid is called simple if it is loopless and all one-element subsets of $e$ are flats.
\item The deletion matroid $M \setminus X$ is a matroid $ (E\setminus X,r_d)$, where $r_d(Y)=r(Y)$, for all subsets 
$Y \subset E\setminus X.$ In particular its nullity function $n_d$ satisfies $n_d(Y)=n(Y).$
\item The contraction matroid $M/X$ is a matroid $ (E\setminus X,r_c)$, where $r_c(Y)=r(X \cup Y)-r(X)$, for all subsets $Y \subset E\setminus X.$ 
\end{enumerate}
\end{definition}

\begin{example} \label{unifo}
Let $M=U_{r,n}$ be the uniform matroid of rank $r$ on an $n$-element set
$
E=\{1,2,\ldots,n\},
$
whose rank function is
\begin{equation*}
r(X)=\min\{|X|,r\}, \qquad X\subseteq E.
\end{equation*}

The basic uniform matroids are as follows.

\begin{enumerate}
\item The nullity of a subset $X\subseteq E$ is
\begin{equation*}
n(X)=|X|-r(X)=
\begin{cases}
0, & |X|\le r,\\
|X|-r, & |X|>r.
\end{cases}
\end{equation*}

\item The independent sets are precisely the subsets of size at most $r$.

\item The circuits are exactly the subsets of cardinality $r+1$.

\item The flats are precisely the subsets $F\subseteq E$ satisfying either
$
|F|<r
$
or
$
F=E.
$

\item If $r>0$, then $U_{r,n}$ has no loops, and if $r<n$, then it has no coloops.

\item For any subset $A\subseteq E$ with $|A|=k\le r$,
$
U_{r,n}\setminus A \cong U_{r,n-k},
$
provided $n-k\ge r$.

\item For any subset $A\subseteq E$ with $|A|=k\le r$,
$
U_{r,n}/A \cong U_{r-k,n-k}.
$
\end{enumerate}
\end{example}

The following is immediate:
\begin{proposition} \label{contractedflats}
 A flat $F_c$ in $M/X$ is formed by taking a flat $F$ in $M$ where $X \subset F$, and setting $F_c= F \setminus X$. In particular $\emptyset$ is a flat of $M/F$ if $F$ is a flat of $M$, so that $M/F$ is loop-free.
 \end{proposition}


\begin{definition}
Given a matroid $M$. Fix a linear ordering on $E(M)$. For a circuit $C$, let $c = \min_{\omega}(C)$. Then the broken circuit is
\begin{equation*}
BC = C \setminus \{c\}.
\end{equation*}
\end{definition}

\begin{definition}
The collection of sets containing no broken circuits 
forms the complex \, $BC_{\omega}(M)$.
\end{definition}
\begin{example}
Let $M=U_{r,n}$ with $E=\{e_1<\cdots<e_n\}$. Since every $(r+1)$-subset of $E$ is a circuit.
Thus the broken circuits are precisely the $r$-subsets whose minimum element is not $e_1$. Hence
\begin{equation*}
\operatorname{BC}(U_{r,n})
=
\{X\subseteq E : |X|\le r-1\}
\cup
\{X\subseteq E : |X|=r,\ e_1\in X\}.
\end{equation*}
\end{example}
\begin{theorem}
The complex $BC_{\omega}(M)$ is shellable.
\end{theorem}

\begin{proof}
Proceed by deletion--contraction induction; the NBC condition is preserved under lexicographic ordering, allowing the shelling to lift from the matroid complex structure. See, e.g.,~\cite{Bjorner,Oxley1992}.
\end{proof}


\begin{definition}
Let $M$ be a matroid, and let $L(M)$ denote the collection of flats of $M$, ordered by inclusion. Then $(L(M), \subseteq)$ forms a lattice, called the \emph{lattice of flats} of $M$. 
A lattice $L$ is called \emph{geometric} if it is isomorphic to $L(M)$ for some matroid $M$.
\end{definition}

\begin{definition} \label{Mob} [Möbius Function on a Poset]
Let $L$ be a finite poset. The Möbius function $\mu : D_{\mu} \subset L \times L \to \mathbb{Z}$ is defined recursively by:
\begin{equation*}
\mu(x,y) =
\begin{cases}
1, & \text{if } y = x, \\
\text{not defined}, & \text{if } x \not\le y, \\
-\displaystyle\sum_{z} \mu(x,z), & \text{if } x < y,
\end{cases}
\end{equation*}
where the sum is over $z$ such that $x \le z < y$.
\end{definition}

\begin{proposition}
Let $L=L(M)$ be the geometric lattice of flats of a matroid $M$. Then the Möbius function $\mu(x,y)$ is well-defined on every interval $[x,y] \subseteq L$, and depends only on the combinatorial structure of the lattice.
\end{proposition}

\begin{definition}
Let $F \in L(M)$ be a flat. The local Whitney number associated with $F$ is
\begin{equation*}
w_F = \mu(\hat{0}, F).
\end{equation*}
\end{definition}
Here $\hat{0}$ is the zero element of the lattice. This is $\emptyset$ if $M$ is loop-free. (Some authors prefer to define the local Whitney number as $|\mu(\hat{0}, F)|=(-1)^{r(F)}\mu(\hat{0}, F)$. We prefer to use the comvention above.)

\begin{definition}
Let $r(F)$ denote the rank of a flat $F \in L(M)$. For each integer $\ell$, the Whitney number $w_\ell$ of the first kind is
\begin{equation*}
w_\ell = \sum_{\{F \in L(M) : \rank(F)=\ell\}} w_F.
\end{equation*}
\end{definition}


\begin{definition} \label{char}
We have: 
\begin{itemize}
    \item {Characteristic Polynomial.} For a finite lattice $L$ of rank $r$,
    \begin{equation*}
    p(L, \lambda) = \sum_{x \in L} \mu(0, x)\,\lambda^{r - r(x)}.
    \end{equation*}
    For a  matroid $M$,
    \begin{equation*}
    p(M, \lambda) = \sum_{X \subseteq E} (-1)^{|X|}\,\lambda^{r(M) - r(X)}.
    \end{equation*}
    If $M$ has a loop, then  $p(M, \lambda)=0.$ If $M$ is loopless, then, by for example \cite{Bjorner}, $p(M, \lambda)=p(L, \lambda)$, for $L$ its lattice of flats.

    \item {Tutte Polynomial.}
    \begin{equation*}
    T_M(X, Y) = \sum_{A \subseteq E} (X-1)^{r(E)-r(A)} (Y-1)^{|A|-r(A)}.
    \end{equation*}
\end{itemize}
\end{definition}
These polynomials are related by
\begin{equation*}
p(M, \lambda) = (-1)^{r(M)} T_M(1-\lambda, 0). 
\end{equation*}

See for example \cite[Formula (7.17)]{Bjorner}.

\begin{proposition} \label{charischrom}
Let $G$ be a graph without coloops, and with $t$ components. Then  $p(M,Z)=Z^{-t}P_{G}(Z),$ for the chromatic polynomial $P_{G}(Z)$  of $G$, for $M$   the cycle matroid of the graph.

Let $q$ be a prime power, and let  $C$ be a $\mathbb{F}_q$-linear code $row(A)$ derived from any incidence matrix $A$ obtained by assigning directions  to the non-loop edges of $G$, with $-1$ for each source, and $1$ for each target. Then $M$ is also the matroid derived from $A$, and from  any generator matrix of $C$. 
\end{proposition}
See for example \cite{Welsh1976,Oxley1992}.


\section{Generalized weight polynomials}
\subsection{Generalized weight polynomials of matroids}
Let $M=(E,r)$ be a matroid on a finite ground set $E$ with $|E|=n$ and nullity function $n^*(X)=|X-r^*(X)$
for its dual matroid $M^*$.

\begin{definition}[Generalized weight polynomials {\cite{JohnsenRoksvoldVerdure}}] \label{weightj}
For a matroid  each $j=0,\dots,n$, define
\begin{equation*}
P_j(Z)
= (-1)^j \sum_{|\sigma|=j} \sum_{\gamma \subseteq \sigma}
(-1)^{|\gamma|} Z^{n^*(\gamma)}.
\end{equation*}
\end{definition}
For the special case $j=n$ we get
\begin{equation*}
P_n(Z)
= (-1)^n \sum_{\gamma \subseteq E}
(-1)^{|\gamma|} Z^{n^*(\gamma)}.
\end{equation*}
\begin{remark} \label{rem} \label{specialcase}
From the definition of the dual matroid we also have: 
\begin{equation*}
P_n(Z)=\sum_{\gamma \subseteq E}
(-1)^{|\gamma|} Z^{r(E)-r(\gamma)}.
\end{equation*}
for the rank function $r$ of $M.$  But this is the characteristic polynomial $p(M,\gamma).$
The equality holds because $n^*(\gamma)=r(E)-r(E \setminus \gamma).$ Passing from  summing over all $\gamma$ to summing over all $E \setminus \gamma$, one must multiply by a factor $(-1)^n.$
\end{remark}





\begin{lemma}
Let $\gamma \subseteq E$ with $|\gamma| \le j$. Then the number of subsets $\sigma \subseteq E$ such that $\gamma \subseteq \sigma$ and $|\sigma|=j$ is
\begin{equation*}
\binom{n - |\gamma|}{j - |\gamma|}.
\end{equation*}
\end{lemma}

\begin{proof}
Any such $\sigma$ is obtained by extending $\gamma$ with a subset of size $j-|\gamma|$ chosen from $E \setminus \gamma$, which has size $n-|\gamma|$.
\end{proof}

\begin{proposition}
The polynomial $P_{j}(Z)$ admits the form
\begin{equation*}
P_j(Z)
=
\sum_{\substack{\gamma \subseteq E\\ |\gamma|\le j}}
(-1)^{j+|\gamma|}
\binom{n-|\gamma|}{j-|\gamma|}
Z^{n^*(\gamma)}.
\end{equation*}
\end{proposition}

\begin{proof}
Exchange the order of summation in the defining expression and apply the reindexing lemma to count admissible extensions $\sigma \supseteq \gamma$ of size $j$:
\begin{align*}
P_j(Z)
&= (-1)^j \sum_{|\sigma|=j} \sum_{\gamma \subseteq \sigma} (-1)^{|\gamma|} Z^{n^*(\gamma)} \\
&= (-1)^j \sum_{\gamma\subseteq E} (-1)^{|\gamma|} Z^{n^*(\gamma)} 
\cdot |\{\sigma:\gamma\subseteq\sigma,\ |\sigma|=j\}| \\
&= (-1)^j \sum_{\gamma\subseteq E} (-1)^{|\gamma|} Z^{n^*(\gamma)} 
\binom{n-|\gamma|}{j-|\gamma|} \\
&=
\sum_{\gamma\subseteq E} (-1)^{j+|\gamma|}\binom{n-|\gamma|}{j-|\gamma|}
Z^{n^*(\gamma)}.
\end{align*}
\end{proof}




\begin{proposition}[Divisibility by $(Z-1)$] \label{factor}
For any matroid $N$ on a ground set $E$ of size $n$, and for every $1 \le j \le n$, the polynomial $P_j(Z)$ is divisible by $(Z - 1)$.
\end{proposition}

\begin{proof}
Recall
\begin{equation*}
P_j(Z)
= (-1)^j \sum_{|\sigma|=j} \sum_{\gamma \subseteq \sigma}
(-1)^{|\gamma|} Z^{n_N^*(\gamma)}.
\end{equation*}
Evaluate at $Z = 1$:
\begin{equation*}
P_j(1)
= (-1)^j \sum_{|\sigma|=j} \sum_{\gamma \subseteq \sigma}
(-1)^{|\gamma|}.
\end{equation*}
For fixed $\sigma$,
\begin{equation*}
\sum_{\gamma \subseteq \sigma} (-1)^{|\gamma|}
= \sum_{k=0}^j \binom{j}{k} (-1)^k
= (1 - 1)^j = 0
\end{equation*}
for $j\ge 1$. Hence $P_j(1)=0$, so $(Z - 1)$ divides $P_j(Z)$.
\end{proof}
\subsection{Connection to linear codes}

Let $C$ be a linear $q$-ary code of length $n$, and $M$ the matroid associated to a generator matrix of $C$.
A main result of \cite{JohnsenRoksvoldVerdure} is:

\begin{lemma}[Weight enumerators as evaluations of $P_j(Z)$]  \label{connection}
Let
\begin{equation*}
C_m := C \otimes_{\F_q} \F_{q^m}.
\end{equation*}
Then, for each $j=0,1,\dots,n$,
\begin{equation*}
A_{C_m,j}=P_j(q^m),
\end{equation*}
where $A_{C_m,j}$ denotes the number of codewords of Hamming weight $j$ in $C_m$.
\end{lemma}

\begin{definition}
For a linear code $C$, let $A_{C,j}(Z)$ be the generalized weight polynomial $P_j(Z)$ for the matroid $m$ associated to a generator matrix of $C$.
\end{definition}


\subsection{Supports, flats, and contractions}

\begin{lemma}[Supports and cycles]\label{lem:supports-cycles}
Let $C_Q = C \otimes_{\F_q} \F_Q$, where $\F_Q/\F_q$ is any finite extension, and $M$ the matroid of a generator matrix of $C$. For a codeword of $C_Q$ with support $\sigma \subseteq E$, the following hold:
\begin{enumerate}
\item $\sigma$ is a \emph{cycle} of $M^*$ (i.e.\ inclusion-minimal among subsets with a given nullity).
\item Equivalently, $F := E \setminus \sigma$ is a flat of $M$.
\end{enumerate}

\end{lemma}

\begin{proof}
Let $H$ be a parity-check matrix for $C$, so its column matroid is $M^*$. If $c \in C_Q$ has support $\sigma$, then 
\begin{equation}
\sum_{e \in \sigma} c_e H_e = 0,
\end{equation}
where every $c_e \neq 0$.

Fix $e \in \sigma$. Assume that  
\begin{equation}
n^*(\sigma \setminus \{e\}) = n^*(\sigma).
\end{equation}
Thus $e$ is a coloop of the restriction $(M^* | \sigma)$. 
Equivalently, every linear relation among the columns indexed by $\sigma$ has coefficient $0$ at $e$. This contradicts the displayed relation, whose coefficient $c_e$ is nonzero.

Hence, for every $e \in \sigma$, 
\begin{equation}
n^*(\sigma \setminus \{e\}) < n^*(\sigma).
\end{equation}
If $Y \subsetneq \sigma$, choose $e \in \sigma \setminus Y$. By monotonicity of nullity under inclusion, 
\begin{equation}
n^*(Y) \le n^*(\sigma \setminus \{e\}) < n^*(\sigma).
\end{equation}
Therefore $\sigma$ is a cycle of $M^*$. By matroid duality, $E \setminus \sigma$ is a flat of $M$.

\end{proof}

\begin{proposition}[Decomposition over cycles / flats]\label{prop:ACj-flats}
Let $C$ be a linear code with associated matroid $M$. Then
\begin{equation}\label{eq:ACj-sum-p}
A_{C,j}(Z)
\;=\;
\sum_{\sigma} p\big( (M^*|\sigma)^*, Z\big)
\;=\;
\sum_{F} p(M/F,Z),
\end{equation}
where
\begin{itemize}
\item the first sum runs over all cycles $\sigma$ of $M^{\ast}$ of cardinality $j$,
\item the second sum runs over all flats $F$ of $M$ of cardinality $n-j$,
\item $M/F$ denotes the contraction of $M$ at $F$.
\end{itemize}
\end{proposition}

\begin{proof}
By definition, $A_{C,j}(Z)=P_j(Z)$ is the generating polynomial counting codewords of $C_Q$ of Hamming weight $j$ as $Q=q^m$ varies. For each cycle $\sigma$ we will count the codewords with support exactly $\sigma$.
For the special case $\sigma=E$ and $j=n$ this is $p(M,Z)$, by Remark \ref{specialcase}.

For an arbitrary cycle $\sigma$, we see that that every codeword that has its support inside $\sigma$ lies in the code $C'$ which is the code obtained by shortening $C_Q$ at $F=E \setminus{\sigma}.$ It is well known that shortening of a code at a subset $F$ gives a new generator matrix, whose associated matroid is the contraction $M/F.$ 
For this matroid $\sigma$ is the ground set and plays the same role for the shortened code as $E$ plays for $C_Q$. Hence the number of codewords in the shortened code with full support $\sigma$ is determined by $p(M/F,Z).$
Summing over all flats of cardinality $n-j$ gives the desired formula.

\end{proof}

We now reformulate Definition \ref{weightj}:
\begin{proposition}  \label{gencontract}
For any matroid we have: \begin{equation*}
P_j(Z)= \sum_{\sigma} p((M^*|\sigma)^*,Z) = \sum_{F: |F|=n-j} p(M/F,Z),
\end{equation*}
\end{proposition}
\begin{proof}
By Definition \ref{weightj}
\begin{equation*}
P_j(Z)
= (-1)^j \sum_{|\sigma|=j} \sum_{\gamma \subseteq \sigma}
(-1)^{|\gamma|} Z^{n^*(\gamma)}.
\end{equation*}
But in this formula only the subsets $\sigma$ that are cycles for $M^*$ contribute. A reason is that if $\sigma$ is not a cycle, then $M^*|\sigma=M^*\setminus F$ has a coloop $e$ (so $r_c(\{e\})=0$, for $r_c$ the rank function of the contraction matroid $M/F,$ and then in each inner sum we sum over all subsets $\gamma$ not containing $e$ and all $\gamma \cup \{e\}$ for these $\gamma.$ Then the terms $(-1)^{|\gamma|} Z^{n^*(\gamma)}$ and 
$(-1)^{|\gamma \cup \{e\}|}Z^{n^*(\gamma \cup \{e\})}$ cancel each other since the two $n^*$-values are the same.
This is true since the nullity function of $M^*$ is the same as the nullity function of $M^* \setminus F$ when restricted to $E\setminus F$, and hence $n^*(\gamma \cup \{e\})-n^*(\gamma)=r_c(E\setminus \gamma) - r_c(E \setminus (\gamma \cup \{e\})) =r_c(Y \cup \{e\})-r_c(Y)$, for $Y=E \setminus (\gamma \cup \{e\}).$ But $r_c(Y \cup \{e\}\le r_c(Y)+r_c(\{e\})=r_c(Y)+0=r_c(Y)$ by axiom (R3) for rank functions.

But, again,  for each fixed flat $F=E \setminus \sigma,$
we have that the nullity function of the restriction matroid $(M^*)|_{\sigma}$ is the same as (the nullity function)  $n^*$ (of the entire $M^*$). Therefore the inner sum $(-1)^j\sum_{\gamma \subseteq \sigma}
(-1)^{|\gamma|} Z^{n^*(\gamma)}$ is identical with the characteristic polynomial of $((M^*)|_{\sigma})^*=(M^*\setminus F)^*=M/F$ by Definition \ref{weightj} and Remark \ref{rem}.
\end{proof}

\begin{remark} \label{characterization}
The converse of Lemma \ref{lem:supports-cycles} does not hold, in the sense that for a linear code over a fixed field not all cycles of the matroid $M^*$ of some parity check matrix necessarily are supports of codewords. An easy example is the binary code generated by $(1,0,1)$ and $(0,1,1)$. Here $M=U(2,3)$, so $M^*=U(1,3),$ and hence we see that $E=\{1,2,3\}$  is a cycle of nullity $2$. But it is not the support of any word in the binary code. But for any non-degenerate linear code the number of points with support $E$ of an extension code of $C$ of cardinality $Q$ is given by a polynomial  $P_n(Q)$, which has only finitely many zeroes.
Hence a cofinite set of the  extensions of the code contain at least one codeword with support $E.$

For each cycle $\sigma$ of $M^*$ the argument concerning a cofinite set of the extensions can be applied to the shortening of the code at $E \setminus \sigma$ in the same way.

In general, for a linear code $C$ these observations give a characterizations of the cycles of $M^*$ as the subsets of $E$ such that for a  cofinite set of extensions of $C$ there is at least one word in that extension code that has support $\sigma$. This clearly is also equivalent to that there exists at least one extension code such that there is at least one word in that code, with support $\sigma.$
\end{remark}

\subsection{Möbius function formulations }

\begin{theorem}[Generalized weight polynomials via Möbius numbers]\label{thm:Mobius}
For any matroid $M$ of rank $r(M)$ on $E$ with $|E|=n$, and for $j=0,\dots,n$,
\begin{equation}\label{eq:Mobius-flats}
P_{j}(Z)
=
\sum_{\substack{F_1 \subseteq F_2 \\ |F_1| = n-j}}
\mu(F_1,F_2)\, Z^{\,\rank(M)-\rank(F_2)},
\end{equation}
where the sum runs over flats $F_1\subseteq F_2$ in $L(M)$ and $\mu(F_1,F_2)$ is the Möbius function of the interval $[F_1,F_2]$.
Equivalently, for cycles of $M^*$ we have:
\begin{equation} \label{eq:Mobius-cycles}
P_{j}(Z) 
=  
\sum_{\substack{\sigma_2\subseteq\sigma_1\\|\sigma_1|=j}}
\mu(\sigma_2,\sigma_1)\, Z^{\,n^*(\sigma_2)},
\end{equation}
where the Möbius function is taken in the poset of cycles of $M^*$ and $n^*(\sigma_2)$ as usual denotes the nullity of $\sigma_2$ for $M^*.$
\end{theorem}

\begin{proof}
We start from
\begin{equation*}
P_{j}(Z) = \sum_{F_1:\,|F_1|=n-j} p(M/F_1,Z),
\end{equation*}
from Proposition~\ref{prop:ACj-flats}. Let $L=L(M)$ and consider the characteristic polynomial of the minor $M/F_1$.
The fact that $F_1$ is a flat gives that $M/F_1$ is loop-free. Hence we may use  usual lattice-of-flats form from Definition \ref{char}:
\begin{equation*}
p(M/F_1,Z) = \sum_{F_{2,c}} \mu(\emptyset,F_{2,c}) Z^{\rank(M/F_1)-\rank_{M/F_1}(F_{2,c})}.
\end{equation*}
Here we sum over all flats $F_{2,c}$ of the contraction matroid $M /F_1$, and $\mu$ is the Möbius function of the lattice of flats of $M/F_1.$  by Definition \ref{basic} ranks in $M/F_1$ satisfy
\begin{equation*}
r_c(F_{2,c}) = r(F_2) - r(F_1),
\end{equation*}
and moreover all flats $F_{2,c}$ of $M /F_1$ are of the form $F_2\setminus F_1$.
We then have 
\begin{equation*}
\rank(M/F_1) - \rank_{M/F_1}(F_{2,c})
=
\big(\rank(M)-\rank(F_1)\big) - \big(\rank(F_2)-\rank(F_1)\big)
=
\rank(M)-\rank(F_2),
\end{equation*}
and $\mu(\emptyset,F_{2,c})$ for the lattice of flats of $M/F_1$ is equal to $\mu(F_1,F_{2})$ for the lattice of flats of $M$, by Proposition  \ref{contractedflats}. All in all, 
\begin{equation*}
p(M/F_1,Z) = \sum_{\substack{F_2\in L\\F_1\subseteq F_2}} \mu(F_1,F_2)\, Z^{\rank(M)-\rank(F_2)}.
\end{equation*}
Now sum over all flats $F_1$ with $|F_1|=n-j$:
\begin{equation*}
P_{j}(Z)
=
\sum_{\substack{F_1:\,|F_1|=n-j}} p(M/F_1,Z)
=
\sum_{\substack{F_1:\,|F_1|=n-j}}
\sum_{\substack{F_2\in L\\F_1\subseteq F_2}}
\mu(F_1,F_2)\, Z^{\rank(M)-\rank(F_2)}.
\end{equation*}
Converting  this double sum to a single sum over all inclusion relations of $2$ flats of $M$, we obtain the statement of the theorem.
For
the second statement note that flats in $M$ correspond to complements of cycles in $M^*$, so $F_i = E\setminus \sigma_i$. This reverses inclusion, %
and $\mu(F_1,F_2)=\mu(\sigma_2,\sigma_1)$, and $r(E)-r(F_2)=n^*(E-F_2)=n^*(\sigma_2)$ by Definition \ref{dual} of the rank function of the dual matroid.
\end{proof}

\subsection{The method of elongations} \label{elonga}
Given a matroid $N$, the $l$'th elongation matroid $N^{(l)}$ of the matroid $N$ with rank function $r$ and nullity function $n$ is the matroid whose nullity function is :
$n^{(l)}(\sigma)=\max(0,n(\sigma)-l),$, for $l=0,\cdots l(E)=|E|-r(E).$  The cycles of $N^{(l)}$ are the same as those of $N$, except that the ones with nullity $1$ for $N$ no longer are cycles for $N^{(l)}$. By convention $N^{(0)}=N.$
In \cite{JohnsenRoksvoldVerdure} one expresses $P_j(Z)$ in terms of Möbius numbers of cycle posets of the elongations $N^{(\ell)}$ of the matroid $N=M^*$:
\begin{equation}  \label{weightelong}
P_j(Z) = \sum_{\ell} (\varphi^{(\ell)}_j - \varphi^{(\ell-1)}_j) Z^{\ell}, \textit{ where}
\qquad
\varphi^{(\ell)}_j = \sum_i (-1)^i\nu^{(\ell)}_{i,j},
\end{equation}
where the $\nu^{(\ell)}_{i,j}$ counts the sum of absolute values of Möbius numbers $\mu(\emptyset,X)$  of fixed size $j$ and nullity $i$  in the $\ell$-th elongation. Here $\mu$ refers to the poset of cycles for $(M^*)^{(\ell)}.$ 

\begin{remark}
Our formula~\eqref{eq:Mobius-flats} is  meant to be "complementary" to this already well-known method of elongations, which we just sketch briefly here: Instead of summing over cycles in elongations, like in \cite[Theorem 5.1]{JohnsenRoksvoldVerdure} and \cite[Theorem 10.4]{MartinezValenciaVillarreal} in a more general setting, we sum directly over \emph{pairs of flats} in the original lattice with a cardinality constraint on the lower flat. Both approaches rely on Möbius functions, but on different posets. Moreover the one in Theorem \ref{thm:Mobius}, as we shall see,  is particularly adapted to Orlik-Solomon algebra and Whitney homology interpretations.
\end{remark}

{\rm One may call the $\nu^{(l)}_{i,j}$ virtual $\mathbf{N}$-graded Betti numbers, following the notation of \cite{JohnsenPratiharVerdureGabidulin}.  There are two results that make the method of elongation particularly applicable: 
By for example \cite[Page 57]{StanleyCM} these numbers, defined  here as Möbius numbers of posets of cycles,  are  in fact equal to  actual $\mathbf{N}$-graded Betti-numbers of minimal resolutions of Stanley-Reisner rings of the  independence complexes of the hierarchy of elongation matroids that we considered above. Concretely:
If $X$ is a cycle with $n^{(l)}(X)=i$ for the matroid $(M^*)^{(\ell)},$ then
$$|\mu(\emptyset,X)|=\nu^{(l)}_{i,X}=\beta^{(l)}_{i,X},$$
for the $\mathbb{N}^{\mathbb{N}}$-graded Betti number $\beta^{(l)}_{i,X}$ of any minimal resolution of the Stanley-Reisner ring of the independence (simplicial) complex of 
$(M^*)^{(\ell)}.$ 

Secondly, in view of this interpretation, the $\phi_j^{(l)}$ satisfy the Boij-Söderberg equations (\cite[Formula (2.1)]{BS}):

\begin{equation}\label{eq:BS-phi}
\sum_{j=0}^nj^s\phi^{(l)}_{j}=0 \quad \textrm{for } s=0,1,\ldots,k-1-l
\end{equation}}
where $k=r(M).$
These equations make it possible to find all the $\phi^{(l)}_{j}$ if one knows just a few of them. 

Summing up this section so far we have: An advantage of Formulas (\ref{eq:Mobius-flats}) and (\ref{eq:Mobius-cycles}), compared to those where elongations are used, is that they only refer to the bivariate Möbius 
numbers of the geometric lattice of flats of a \underline{single} matroid. An advantage of \cite[Theorem 5.1]{JohnsenRoksvoldVerdure}, and that of \cite[Theorem 10.4]{MartinezValenciaVillarreal}, compared with Formulas  (\ref{eq:Mobius-flats})  and (\ref{eq:Mobius-cycles}) is that all Möbius numbers one uses are essenially univariate and  of type $\mu(0,x)=\mu(x)$, for various posets of cycles, and that these numbers  are also Betti numbers of Stanley-Reisner rings, and under fortunate circumstances can be more easily found using algebraic techniques,  and well-known identities between such Betti numbers. It is not a main point of the present paper to use the method of elongations,
but rather to demonstrate that there also are other ways to find the generalized weight polynomials.

\section{The Orlik--Solomon algebra}

\begin{definition}
An atom of a matroid is a flat of rank $1$. 
Let $M$ be a matroid with  set of atoms $A(M)$. The exterior algebra $\Ecal$ over a field $\mathbb{K}$ of a geometric lattice $L=L(M)$ is
\begin{equation*}
\Ecal = \bigwedge (e_a \mid a \in A(M)),
\end{equation*}
where the generators satisfy $e_a \wedge e_b = - e_b \wedge e_a$ and $e_a^2 = 0$.
\end{definition}

\begin{definition}
For a monomial $e_S = e_{s_1} e_{s_2} \cdots e_{s_\ell} \in \Ecal$, define the linear map $\delta$ by
\begin{equation*}
\delta(e_S) = \sum_{i=1}^{\ell} (-1)^{i-1} e_{s_1} \cdots \widehat{e_{s_i}} \cdots e_{s_\ell},
\end{equation*}
where $\widehat{e_{s_i}}$ denotes omission of the factor. Extend $\delta$ linearly to all of $\Ecal$.
\end{definition}

\begin{definition}
Let $\mathcal{C}(M)$ denote the collection of circuits (inclusion minimal dependent subsets $C$ of atoms, where a dependent set of atoms is a set  spanning a flat of smaller rank than $|C|$). The Orlik--Solomon ideal $\Ical$ is
\begin{equation*}
\Ical = \left\langle \delta(e_S) : S \in \mathcal{C}(M) \right\rangle \subseteq \Ecal.
\end{equation*}
\end{definition}

\begin{definition}
The Orlik--Solomon algebra of $M$ is the graded quotient algebra
\begin{equation*}
\OS(M) = \Ecal / \Ical.
\end{equation*}
\end{definition}

\begin{definition}
We recall that a simple matroid is a loop-less matroid with all singletons sets being flats, so that there are no circuits of cardinality two. The simplification $M_S$ of a matroid $M$ is a (simple) matroid with ground set $A(M)$, where the rank of a subset $B$ of $A(M)$ is the smallest $M$-rank of a flat containing the set in $E(M)$ corresponding to the elements forming $B$.
\end{definition}

\begin{remark}
By \cite[Proposition 7.4.1.]{Bjorner} the simplicial complexes $BC_{\omega}(M)$ and $BC_{\omega}(M_s)$ are isomorphic. 
By \cite[Theorem 7.4.6]{Bjorner} this also implies that the characteristic polynomial of a loopless matroid is equal to that of its simplification. For a simple matroid  the ground set $E(M)$ and the set of atoms $A(M)$ can be identified, and $\mathcal{C}(M)$ can be identified with the set of circuits of the matroid $M$ in the usual sense.
\end{remark}

\begin{proposition}[Flat-wise OS decomposition and Möbius values] \label{prop:OS-decomp}
Let $L=L(M)$ be the lattice of flats of a loop-free matroid $M$, and $\OS(M)$ its Orlik--Solomon algebra over a field $K$.
\begin{enumerate}
\item As a $K$-vector space,
\begin{equation*}
\OS(M) = \bigoplus_{F \in L} \OS_F,
\end{equation*}
where $\OS_F$ is spanned by NBC monomials $e_S$ with closure $\operatorname{cl}(S)=F$.
\item For each flat $F$,
\begin{equation*}
\dim_K \OS_F = (-1)^{\rank(F)} \mu(\hat{0},F),
\end{equation*}
where $\mu$ is the Möbius function of $L$.
\end{enumerate}
\end{proposition}

\begin{proof}
The decomposition is exactly~\cite[(2.10)]{OrlikSolomon}: NBC sets (those containing no broken circuit) are partitioned by closure, yielding $\OS(M)=\bigoplus_F \OS_F$. For (2), the order complex of the interval $(\hat{0},F)$ is shellable and hence homotopy equivalent to a wedge of spheres of dimension $\rank(F)-2$, say $\bigvee_{i=1}^{t_F} S^{\rank(F)-2}$, with $t_F = |\mu(\hat{0},F)|$ by standard poset-topology results~\cite{Bjorner,StanleyCM}. Therefore
\begin{equation*}
h_{\rank(F)-2}((\hat{0},F)) = |\mu(\hat{0},F)| = (-1)^{\rank(F)} \mu(\hat{0},F),
\end{equation*}
for this topology, since $\mu(\hat{0},F)$ has sign $(-1)^{\rank(F)}$ in a geometric lattice. The NBC sets supported in $F$ index a basis of the homology spaces  $\widetilde{H}_{\rank(F)-2}(\emptyset,F)$ over some suitable field, hence
\begin{equation*}
\dim_K \OS_F = \dim \widetilde{H}_{\rank(F)-2}((\hat{0},F)) = (-1)^{\rank(F)} \mu(\hat{0},F).
\end{equation*}
For more details about these homology spaces, see Subsection \ref{homol}.
\end{proof}
\begin{corollary}[Graded OS dimensions and Möbius numbers]\label{cor:OS-graded-Mobius}
For each $\ell$,
\begin{equation*}
\OS^\ell(M) = \bigoplus_{\rank(F)=\ell} \OS_F,
\quad
\dim \OS^\ell(M) = \sum_{\rank(F)=\ell} (-1)^{\ell} \mu(\hat{0},F).
\end{equation*}
\end{corollary}

\begin{theorem}[OS realization of single-flat Möbius data]\label{thm:OS-mobius}
Let $M$ be a loop-free  matroid with lattice of flats $L(M)$ and Orlik--Solomon algebra $\OS(M)$. Then for every flat $F \in L(M)$,
\begin{equation*}
\dim \OS_F
=
(-1)^{\rank(F)} \mu(\hat{0},F).
\end{equation*}
In particular, the graded dimensions of $\OS(M)$ determine the characteristic polynomial:
\begin{equation*}
p(M,Z)
=
\sum_{k=0}^{r(M)} (-1)^k \dim \OS^k(M)\, Z^{r(M)-k}.
\end{equation*}
\end{theorem}

\begin{proof}
The first statement is Proposition~\ref{prop:OS-decomp}(2). For the characteristic polynomial,
\begin{equation*}
p(M,Z) = \sum_{F \in L(M)} \mu(\hat{0},F) Z^{r(M)-\rank(F)}.
\end{equation*}
Grouping by rank $k=\rank(F)$ and using Corollary~\ref{cor:OS-graded-Mobius} yields
\begin{equation*}
p(M,Z)
=
\sum_{k=0}^{r(M)} \left(\sum_{\rank(F)=k} \mu(\hat{0},F)\right) Z^{r(M)-k}
=
\sum_{k=0}^{r(M)} (-1)^k \dim \OS^k(M)\, Z^{r(M)-k}.
\end{equation*}
\end{proof}

\begin{corollary}[OS control of the top generalized weight polynomial]\label{cor:OS-top-weight}
Let $C$ be a non-degenerate linear code with associated (and thus loop-free) matroid $M$. Then the top generalized weight polynomial satisfies
\begin{equation*}
P_{n}(Z)
=
p(M,Z)
=
Z^{r(M)} \sum_{k=0}^{r(M)} (-1)^k \dim \OS^k(M)\, Z^{-k}.
\end{equation*}
\end{corollary}

\begin{proof}
The equality $P_{n}(Z)=p(M,Z)$ is standard in the matroid representation of codes; see e.g.~\cite{JohnsenRoksvoldVerdure,Greene1976}. The OS-expression for $p(M,Z)$ is given by Theorem~\ref{thm:OS-mobius}. Factoring out $Z^{r(M)}$ yields the claimed form.
\end{proof}

\begin{remark}
Furthermore we have (\cite[(2.18)]{OrlikSolomon}):
There is an exact sequence 
\begin{equation} \label{exact}
0 \rightarrow OS^{r(M)} \rightarrow OS^{r(M)-1}\rightarrow  \cdots \rightarrow OS^1 \rightarrow OS^0 \rightarrow 0.
\end{equation}
As an immediate consequence we also obtain:

$\Sigma_{l=0}^{r(M)}(-1)^l \dim OS^l=0, $ or put in another way, using Corollary \ref{cor:OS-top-weight}: 
$P_{n}(1)=0, $ so this gives another proof of Proposition \ref{factor}, i.e. that $Z-1$ is a factor in $P_{n}(Z),$ and  then of all $P_{j}(Z)$, by construction.
\end{remark}

\begin{remark} \label{practical}
Introducing the Orlik-Solomon algebra in order to understand generalized weight polynomials may seem like 
superfluous theoretization with no practical purpose. But a main  point is, as we will see in Subsection \ref{Using the OS-algebra, method with NBC-sets}, that since each $\OS^k(M)$  has a basis of NBC sets of cardinality $k$, as stated above, determining the coefficients of the characteristic polynomial amounts to counting NBC sets of each cardinality. And this can be a very practical tool, quite different from other methods.
\end{remark}

\subsection{Poincaré polynomials and characteristic polynomials}

\begin{definition}
For a finite-dimensional graded algebra
\begin{equation*}
S=\bigoplus_{i=0}^m S_i,
\end{equation*}
its \emph{Poincaré polynomial} is
\begin{equation*}
P_S(Z)=\sum_{i=0}^m (\dim S_i)\, Z^i.
\end{equation*}
For the Orlik--Solomon algebra $\OS(M)$, we obtain
\begin{equation*}
P_{\OS}(Z)
 = \sum_{\ell=0}^{r(M)} (\dim \OS^\ell(M))\, Z^\ell
 = \sum_{F\in L(M)} (-1)^{\rank(F)}\mu(0,F)\, Z^{\rank(F)},
\end{equation*}
using Proposition~\ref{prop:OS-decomp}.
\end{definition}

\begin{lemma}\label{lem:pL-PA}
For the lattice $L(M)$ of flats of a loop-free matroid $M$, we have
\begin{equation*}
p(L(M),Z) = Z^{r(M)}\, P_{\OS}\!\left(-\frac{1}{Z}\right).
\end{equation*}
\end{lemma}

\begin{proof}
Write $P_{\OS}(t)=\sum_{k=0}^{r(M)} d_k t^k$. Then
\begin{equation*}
Z^{r(M)} P_{\OS}\!\left(-\frac{1}{Z}\right)
= \sum_{k=0}^{r(M)} (-1)^k d_k Z^{r(M)-k}.
\end{equation*}
By Proposition~\ref{prop:OS-decomp}, $d_k = \sum_{\rank(F)=k} (-1)^k \mu(0,F)$, so
\begin{equation*}
Z^{r(M)} P_{\OS}\!\left(-\frac{1}{Z}\right)
=
\sum_{k=0}^{r(M)} \left(\sum_{\rank(F)=k} \mu(0,F)\right) Z^{r(M)-k}
= \sum_{F \in L(M)} \mu(0,F)\, Z^{r(M)-\rank(F)},
\end{equation*}
which is $p(L(M),Z)$.
\end{proof}

\begin{corollary}\label{cor:top-P-A}
For a matroid $M$ associated to a non-degenerate linear code $C$,
\begin{equation*}
P_{n}(Z) = A_{C,n}(Z) = Z^{r(M)} P_\OS\!\Big(-\frac{1}{Z}\Big).
\end{equation*}
\end{corollary}

\begin{proof}
We have $P_{n}(Z)=p(M,Z)$, and $p(M,Z)=p(L(M),Z)$. Lemma~\ref{lem:pL-PA} gives the desired expression. The equality with $A_{C,n}(Z)$ follows from the matroid representation of $C$.
\end{proof}

\subsection{Weight polynomials via Orlik--Solomon}

\begin{proposition}[Weight polynomials via Orlik--Solomon]\label{prop:OS-weight-general}
Let $M$ be a matroid of rank $r(M)$ on $E$ and $M^*$ its dual. For each $j=0,\dots,n$,
\begin{equation*}
P_j(Z) =\sum_{\substack{F \subseteq E \\ |F|=n-j}}  \sum_{k=0}^{\rank(M/F)} (-1)^k \dim \OS^k(M/F)\, Z^{\rank(M/F)-k}.
\end{equation*}
\end{proposition}

\begin{proof}
We have:
\begin{equation*}
P_{j}(Z) = \sum_{F:\,|F|=n-j} p(M/F,Z),
\end{equation*}
where $p(M/F,Z)$ is the characteristic polynomial of the contraction $M/F$.

For each such flat $F$, $M/F$ is a matroid of rank $\rank(M/F) = \rank(M)-\rank(F)$. The characteristic polynomial of $M/F$ can be expressed in terms of the Orlik--Solomon algebra $\OS(M/F)$ by Theorem~\ref{thm:OS-mobius}:
\begin{equation*}
p(M/F,Z) = \sum_{k=0}^{\rank(M/F)} (-1)^k \dim \OS^k(M/F)\, Z^{\rank(M/F)-k}.
\end{equation*}
Then we are left with the stated formula
\begin{equation*}
P_j(Z) =\sum_{\substack{F \subseteq E \\ |F|=n-j}}  \sum_{k=0}^{\rank(M/F)} (-1)^k \dim \OS^k(M/F)\, Z^{\rank(M/F)-k}.
\end{equation*}
\end{proof}

\subsection{Deletion--contraction sequences for OS algebras}
Given a geometric lattice $L$ and an atom $e$. The contraction $L/ \{e\}$ is the interval $[\{e\},\hat{1}].$
The deletion $L \setminus \{e\}$ is the lattice formed by combining all atoms except $\{e\}.$
An atom is an isthmus if its removal drops the rank of the lattice. This corresponds to a coloop of the associated simple matroid.

For completeness we recall a well-known result and its proof:
\begin{theorem}[Deletion--Contraction Exact Sequence {\cite{OST1984}}]
Let $L=L(M)$ be a geometric lattice and let $e$ be an atom which is not an isthmus . 
Then there exists a short exact sequence of Orlik--Solomon algebras:
\begin{equation*}
0 \longrightarrow \OS(L-e) \xrightarrow{i} \OS(L) \xrightarrow{j} \OS(L/e) \longrightarrow 0.
\end{equation*}
\end{theorem}

\begin{proof}
Fix a linear order on the set of atoms $A(L)$ such that $e$ is minimal. Let $E(L)$ be the exterior algebra on $A(L)$, and $\OS(L)=E(L)/I(L)$ be the Orlik--Solomon algebra. Similarly define $\OS(L-e)$ and $\OS(L/e)$.

Define
\begin{equation*}
i:\OS(L-e)\rightarrow \OS(L), \quad e_S \mapsto e_S.
\end{equation*}
Define $j$ on generators by
\begin{equation*}
j(e_S)=
\begin{cases} 
e_{\lambda(S\setminus\{e\})}, & e\in S, \\
0, & e\notin S,
\end{cases}
\end{equation*}
where $\lambda(H)=e\vee H$ in $L$. This induces a well-defined graded algebra homomorphism $j:\OS(L)\rightarrow \OS(L/e)$.

By Proposition \ref{prop:OS-decomp} 
$\OS(L-e)$ has a basis indexed by NBC-sets not containing $e$, which embed into the corresponding basis of $\OS(L)$; hence $i$ is injective.

Let $e_S$ be a basis element of $\OS(L)$. If $e\notin S$, then $j(e_S)=0$. If $e\in S$, then $j(e_S)\neq 0$. Hence
\begin{equation*}
\ker(j)=\mathrm{span}\{e_S \mid e\notin S\} = \operatorname{im}(i).
\end{equation*}
Surjectivity of $j$ follows since every basis element of $\OS(L/e)$ is represented by some $e_{\lambda(S')}$ with $S'\subseteq A(L-e)$, and then $e_{S'\cup\{e\}}$ maps to it. The sequence is exact.
\end{proof}
\begin{example}
Let the uniform matroid $M=U_{2,4}$
The deletion--contraction theorem gives the short exact sequence

\begin{equation*}
0 \longrightarrow \OS(U_{r,n-1})
  \longrightarrow \OS(U_{r,n})
  \longrightarrow \OS(U_{r-1,n-1})
  \longrightarrow 0,
\end{equation*}

and in particular,

\begin{equation*}
0
\longrightarrow
\OS(U_{2,3})
\longrightarrow
\OS(U_{2,4})
\longrightarrow
\OS(U_{1,3})
\longrightarrow
0.
\end{equation*}
\end{example}

An analogous result holds at the level of Whitney homology (See Subsection \ref{homol}) for definitions).

 \section{Homology numbers}  \label{homol}
 \subsection{Order homology}

The order, or Whitney, homology of a matroid is typically studied through a simplicial complex $S_L$ derived from its lattice of flats $L$. The vertex set of $S_L$ consists of all flats of $M$ excluding the minimal element ($\hat{0}$) and the maximal element ($\hat{1}$). The simplices are formed by all totally ordered chains of these flats, meaning the facets of $S_L$ correspond to maximal chains of length $r-1$, where $r = \rank(M)$. Consequently, these facets have dimension $r-2$, and for a geometric lattice, the homology is concentrated entirely in this top dimension. Specifically, the top homology dimension $h_{r-2}(S_L)$ is equal to the absolute value of the Möbius function $|\mu(L)|$.

Whitney homology refines this global topological view by collecting the local topology of every principal lower interval in the lattice. It is defined as:
\begin{definition}
Let $L = L(M)$ be the lattice of flats of a loop-free matroid $M$. The Whitney homology is the direct sum
\begin{equation*}
H^{W}(L) = \bigoplus_{F \in L} \widetilde{H}_\bullet(\hat{0},F)
\end{equation*}
where $\widetilde{H}_\bullet(\hat{0}, F)$ denotes the reduced homology of the order complex associated with the open interval $(\hat{0}, F)$.
\end{definition}
This collection of local homology groups is $\mathbb{N}$-graded by rank, such that the $l$-th Whitney homology group is $H^W(L)_l = \bigoplus_{\rank(F)=l} H_{\rank(F)-2}(S_F)$. The dimension of this $l$-th group is the absolute value of the Whitney number of the first kind, $|w_l| = \sum_{\rank(F)=l} |\mu(\hat{0}, F)|$.

A fundamental result in this framework is the Orlik--Solomon--Whitney Isomorphism, which states that the Whitney homology ring is isomorphic as a graded algebra to the Orlik--Solomon algebra $\OS(M)$. This isomorphism provides a bridge between topological invariants and combinatorial structures, as the graded dimensions of the OS algebra (and thus Whitney homology) are indexed by No Broken Circuit (NBC) sets.

\begin{theorem}[Orlik--Solomon--Whitney Isomorphism]\label{thm:OS-Whitney}
Let $M$ be a matroid with lattice of flats $L(M)$. There exists a graded algebra isomorphism
\begin{equation*}
H^{W}(L(M)) \cong \OS(M),
\end{equation*}
where $H^W(L(M))$ is the Whitney homology ring and $\OS(M)$ is the Orlik--Solomon algebra.
\end{theorem}

This result is given in \cite[Theorem 7.10.2]{Bjorner}. It also follows from \cite{Bjorner} that the graded bit nr. $k$ for each of the rings is generated by NBC-set of cardinality $k$, for each $k\in [0,r(M)].$ This will be a useful tool to calculate generalized weight polynomials.
The shellability of each interval $(\hat{0},F)$ in the geometric lattice $L(M)$ is classical; see Orlik--Solomon~\cite{OrlikSolomon} and Bj\"orner~\cite{Bjorner}. 
Compatibility of the products (join of flats vs.\ wedge product) is established in~\cite{OrlikSolomon,OST1984}. Matching these structures yields the ring isomorphism.

It is then well known that for the dimensions $h_i(S_L)$ of the (reduced) homology groups $H_i(S_L)$ over $\mathbf{Z}$ we have:

\begin{proposition} \label{homWhit}
$h_{r(M)-2}(S_L)=(-1)^{r(M)}\mu(L)=|\mu(L)|$, and $h_i(S_L)=0,$ for $i \ne r(M)-2$.
\end{proposition}
Moreover this is true in "microcosmos" also:
$h_{r(F)-2}(S_F)=(-1)^{r(F)}\mu(L_F)=|\mu(F)|$, and $h_i(S_F)=0,$ for $i \ne r(F)-2$,
for each flat $F$, where $S_F$ is the corresponding simplicial complex obtained from the lattices interval $L_F=[\emptyset,F].$
For $F$ the zero flat, one sets $h_{-2}(S_F)=1.$
Set-theoretically, one then has:
\begin{definition}
$H^W(L)=\oplus_{F \in L}H_{r(F)-2}(S_F).$

The local  Whitney number   $w_F$  is $(-1)^{r(F)} h_{r(F)-2}(S_F)=\mu(\emptyset, F).$
One also sets: 
$H^W(L)_l=\oplus_{r(F)=l}H_{r(F)-2}(S_F).$

The Whitney number (of the first kind)  $w_l$  is 

$rank H^W(L)_l=\Sigma_{r(F)=l}h_{r(F)-2}=
\Sigma_{r(F)=l}(-1)^l \mu(\emptyset,F).$ We call this number $h_l.$

\end{definition}
Furthermore one has: $H^W(L)=\oplus_{l=0}^{r(M)}H^W(L)_l$.

It turns out that when $H^W(L)$ is $\mathbf{N}$-graded by $l$ in this way, with a natural sum inherited from its homology group constituents, one can define a product on $H^W(L)$ which makes it anti-commutative graded ring isomorphic to the Orlik-Solomon algebra as graded rings.

For each flat $F$ of a matroid $M$, let $w_{F,k}$ be the Whitney number of the first kind
for the matroid $M/ F$ (or $((M^*)|_{\sigma})^*$, where $\sigma=E \setminus F$ is a cycle of $M^*.$)

The set-theoretic part of the statements above give (when rewriting the definition of the characteristic polynomial of a geometric lattice, and also  Equation (\ref{eq:Mobius-flats}), in a language using (Whitney) homology numbers):

\begin{proposition} \label{Whitney}
The $n$'th generalized weight polynomial satisfies:
  $$P_{n}(Z)=\Sigma_F (-1)^{r(F)} h_{r(F)-2}(S_F)Z^{r(M)-r(F)},\textrm{ and }$$ 

  $$P_{n}(Z)=\Sigma_{l=0}^{r(M)} (-1)^lh_l Z^{r(M)-l},$$
  where $h_l$ is the sum of the $h_{r(F)-2}$-values for the flats of rank $l.$
We then also have the more general formula: 
  \begin{equation}  \label{onlywhitney}
P_{j}(Z)=\Sigma_{|F|=n-j}\Sigma_{l=0}^{r(M)-r(F)}(-1)^l h_{F,l}Z^{r(M)-r(F)-l}
\end{equation}
where $h_{F,l}$ is the $h_l$-term for the contraction matroid $M/F.$ 
 If $M$ is derived from the generator matroid of a linear code $C$, then 
   $A_j(C_Q)=P_{j}(Q)$ and given by these formulas.
\end{proposition}
\begin{corollary} \label{newtool}
The first result of Proposition \ref{Whitney} remains true if the local homology numbers $h_{r(F)-2}$ are regarded as  the absolute value of the reduced Euler characteristics of the simplicial complexes $S_F$. The same is true for the analogues of these numbers for the contraction matroids.  
\end{corollary}

\begin{proof}
As we have seen, all these simplicial complexes are shellable (\cite{Bjorner})), and all reduced homology of the $S_F$ is then concentrated at the top level,
so the absolute value of the reduced Euler characteristic is then equal to the value of a single non-zero homology number, i.e plus minus the local Whitney number in question.
\end{proof}
\begin{remark} \label{homolo}
{\rm Corollary \ref{newtool} gives a new tool to calculate the homology numbers, since they up to sign are equal to  the alternating sum of the face numbers of the simplicilal complexes appearing. We will use this tool in Subsection \ref{usingorder}.  The absolute values of the local Whitney numbers  may even be seen  as certain Betti numbers associated to the Stanley-Reisner rings of the simplicial complexes $S_F$. This may be seen by using Hochster's formula in a proper way.
It is unclear to which extent this may be helpful in determining them.} 
\end{remark}

\subsection{Hyperplane arrangements and cohomology}
The following subsection is optional for coding theorists; it applies to matroids that are representable over complex numbers, and we include it to complete the picture. 
\begin{definition}[Hyperplane arrangement and matroid]\label{def:hyperplane-setup}
Let $\mathcal{H} = \{H_1,\dots,H_n\}$ be an arrangement of hyperplanes in $\C^\ell$. Each $H_i$ is given by a linear equation
\begin{equation*}
a_{1} x_1 + \cdots + a_{\ell} x_\ell = 0,
\end{equation*}
and the columns $(a_1,\dots,a_\ell)^T$ form a matrix $A$ defining a matroid $M=M(\mathcal{H})$.
Hence the codimension of the intersection of a set of these hyperplanes defines the rank of the set of hyperplanes.

\begin{equation*}
T = \C^\ell \setminus \bigcup_{i=1}^n H_i
\end{equation*}
be the complement.
\end{definition}

A fundamental theorem of Orlik and Solomon states that the cohomology ring of the arrangement complement $T$ depends only on the underlying matroid $M$. 

More precisely, the cohomology ring $H^\ast(T)$ is naturally isomorphic to the Orlik--Solomon algebra $\OS(M)$, a graded algebra defined purely from the combinatorial data of $M$. 

This result establishes a remarkable bridge between the combinatorics of matroids and the topology of hyperplane arrangement complements.

\begin{theorem}\label{thm:orlik-solomon}
(\cite[Theorem 5.2]{OrlikSolomon}). For every complex arrangement $\mathcal{H}$, the singular cohomology algebra of
$T$ with $\mathbb{Z}$-coefficients is isomorphic, as a graded algebra,
to the Orlik--Solomon algebra of its matroid:
\begin{equation*}
H^{*}\bigl(T, \mathbb{Z}\bigr)
\cong
OS\bigl(M\bigr).
\end{equation*}
\end{theorem}

We give a more detailed explanation:
\begin{definition}[Differential forms and OS algebra]\label{def:R-omega}
One constructs differential forms $\omega_B$ on $T$ associated with sets $B\subseteq[n]$ (typically logarithmic forms $\omega_i=\mathrm{d}\log f_i$ for defining equations of $H_i$, and wedge them). Define a graded algebra
\begin{equation*}
R = \bigoplus_{p=0}^\ell R^p,
\qquad R^p = \mathrm{span}\{\omega_B : |B|=p\}.
\end{equation*}
\end{definition}

\begin{theorem}[Orlik--Solomon algebra = de Rham cohomology ring {\cite{OrlikSolomon,OST1984}}]\label{thm:OS-deRham}
The following graded $K$-algebras are isomorphic:
\begin{enumerate}
\item The algebra $R$ spanned by the $\omega_B$;
\item The Orlik--Solomon algebra $\OS(M)$ of the matroid $M$;
\item The de Rham cohomology ring $H^\ast(T)$ of the arrangement complement.

\end{enumerate}
\end{theorem}

\begin{proof}
The generators $\omega_B$ satisfy the same antisymmetry and circuit boundary relations that define the Orlik--Solomon algebra, giving $R \cong \OS(M)$. The Orlik--Solomon theorem identifies $\OS(M)$ with $H^\ast(T)$ via the logarithmic de Rham complex, mapping generators to logarithmic differential forms and respecting grading and wedge product. Composing these identifications gives the chain of isomorphisms.
\end{proof}




\begin{definition}\label{lem:characteristic-polynomial}
Let $\mathcal{A}$ be an essential arrangement of hyperplanes in $\mathbb{R}^d$ (meaning that the intersection of all hyperplanes is $\{0\})$, and let
$L(\mathcal{A})$ denote its intersection lattice. The \emph{characteristic polynomial}
of $\mathcal{A}$, equivalently of $L=L(\mathcal{A})$, is defined by
\begin{equation*}
p(\mathcal{A};\lambda)
=
p(L;\lambda)
=
\sum_{x\in L}
\mu(\hat{0},x)\,\lambda^{\,d-r(x)},
\end{equation*}
where $\mu$ is the Möbius function of the lattice $L$, and $r(x)$ denotes the rank of
the element $x$.
\end{definition}

As one sees, the characteristic polynomial is closely related to the topology of the
complexified arrangement $\mathcal{A}^{\mathbb{C}}$. By
Theorem~\ref{thm:orlik-solomon}, $p(\mathcal{A};\lambda)$ can be obtained by using Lemma \ref{lem:pL-PA} on the Poincaré polynomial of the cohomology algebra of the complement
$\mathbb{C}^d\setminus\mathcal{A}^{\mathbb{C}}$.

\section{An example for illustration}

\subsection{Using properties of a graph} \label{square}
As we have seen in Proposition \ref{charischrom}, when $G$ is  a graph without coloops, and with $t$ components, and $n$ edges, and $C$ is  the $q$-ary code derived from any incidence matrix obtained by assigning directions  to the non-loop edges of $G$, for any prime power $q$, then the $n$'th weight polynomial $P_n(Z)$ of $C$ is equal to the characteristic polynomial of the cycle matroid of $G$, which again is equal to  $Z^{-t}P_{G}(Z),$ for the chromatic polynomial $P_{G}(Z)$  of $G.$

The following result will be useful, and is most probably well known, but for the benefit of the reader we will give a proof below:
\begin{lemma} \label{complete}
Let $M$ be the cycle matroid of the complete graph $G=K_m$, for some natural number $m$.
Then the flats of $M$ are the edge sets $E(G')$ of cycle matroids of the subgraphs $G'$ of $G$, which are such that each connected component of $G'$ is a complete graph $K_s$ for some natural number $s$. By convention, by a subgraph, we mean a graph that has the same vertex set as $G$, and an edge set which is a subset of that of $G$.

Moreover, if one contracts $G$ in a flat, which thus corresponds to a partition of type $m_1,\cdots, m_s$ of its vertex set $V$, then the simplification of the contraction (the contraction may be a multigraph)  will be a new complete graph $K_a$, where $a= m -\Sigma_{i=1}^s (m_i-1) .$
\end{lemma}

\begin{proof}
The rank of a subgraph $G'$ is $v-c'$, where $v=m$, and $c'$ is the number of connected components of $G'$.
If $G'$ is not as described in the lemma, then it has a connected component with two vertices that are not connected by an edge. Add that edge $e$, and $v$ and $c'$ remain unaltered. Hence $r(E(G'))=r(E(G') \cup \{e\},$ and $E(G')$ is not a flat.
If $G'$ on the other hand  is as described, and not equal to $G$ itself, we may add an edge $e$ to $E(G')$.
Then, by the convention  in the lemma, $e$ connects two components of $G'$ (including the possiblity of a component being an isolated point $K_1 $), and $c'$ decreases by $1$, and the rank consequently increases by $1$. Hence $E(G')$ 
is a flat). If $G'=G$, then $E(G')=E(G)$ is a flat by definition.

The last statement is obvious, since the contraction consists of deleting a certain number of edges and then perform $\Sigma_{i=1}^s (m_i-1)$ point identifications afterwards. After these point identifications all pairs of points are connected by by at least one edge, since they were so before the identifications took place. 
\end{proof}
\begin{remark}
    The first part of Lemma \ref{complete} can be given a natural generalization to all graphic matroids. 
    
    The isolated points of $G'$ have no influence on the cycle matroid of $G'$, so in the rest of this subsection, when we work with flats of the cycle matroid $K_4$, we will disregard the components of type $K_1$ when describing the flats.

    Moreover the lattice of flats of a cycle matroid of a multigraph is isomorphic to the lattice of flats of the cycle matroid of the  graph which is the  simplification of the multigraph (or equivalently: It is isomorphic to  the lattice of flats of the simplification of the cycle matroid of the multigraph).
    
    \end{remark}

\begin{corollary}
For $m \ge 1$ the number of flats for the cycle matroid of $K_m$ is the Bell number  $B_m$, which counts the disjoint partitions of the vertex set. $B_0=1$ by convention, and $B_{n+1}=\Sigma_{k=0}^n \binom{n}{k}B_k$ for 
$n \ge 1$.                                                        
\end{corollary}

\begin{example} \label{contractionexample}

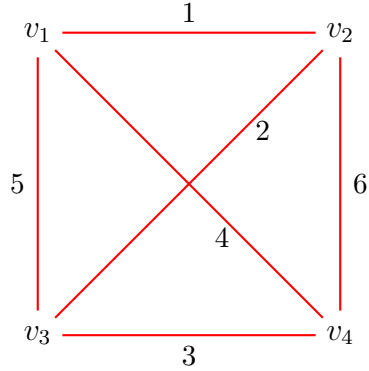
\begin{figure}[ht]
\centering
\begin{tikzpicture}[node distance=2.5cm, every node/.style={circle, draw=none, inner sep=2pt}]
  \node (v1) at (-2, 2)  {$v_1$};
  \node (v2) at (2, 2)   {$v_2$};
  \node (v3) at (-2, -2) {$v_3$};
  \node (v4) at (2, -2)  {$v_4$};

  \draw [red, thick] (v1) -- (v2) node[midway, above, black] {\small $1$};
  
  \draw [red, thick] (v2) -- (v3) node[pos=0.3, right, black] {\small $2$};
  
  \draw [red, thick] (v3) -- (v4) node[midway, below, black] {\small $3$};
  
  \draw [red, thick] (v4) -- (v1) node[pos=0.3, left, black] {\small $4$};
  
  \draw [red, thick] (v1) -- (v3) node[midway, left, black] {\small $5$};
  
  \draw [red, thick] (v2) -- (v4) node[midway, right, black] {\small $6$};

\end{tikzpicture}
\caption{The complete graph $K_4$ associated with the matroid $M$.}
\label{figure1}
\end{figure}

{\rm 

We will now look at the complete graph $K_4$ drawn in Figure~\ref{figure1}. 
Its incidence matrix $A_q$ over $\mathbf{F}_q$  after choosing a direction for each of the $6$ edges, can be row reduced to $3 \times 6$ generator matrices $G_q$ and for a linear $q$-ary code $C^q$. Let $M$ be the matroid, which simultaneously is the cycle matroid of $K_4$, and the matroids obtained from $G_q.$ We want to find the $P_j(Z),$ for $j=0,\cdots,6.$

The $B_4=15$   flats of $M$ then correspond to subgraphs of $K_4$. There are 5 kinds of flats:
\begin{itemize}
\item The entire graph $K_4$ with $6$ edges and rank $3=|V|-1.$
\item Its $4$ complete subgraphs of type $K_3$, each with 3 edges and rank $2.$
\item $3$ pairs of $2$ edges, where the edges in each pair do not meet in a vertex, each with rank $2$.
\item $6$ graphs with just $1$ edge each, all of them of rank $1$.
\item A graph with no edges (or the graph with $6$ vertices and no edges if one prefers).
\end{itemize}
These correspond, taking complements,  to cycles of $M^*$ of cardinalities $0,3,4,5,6$ respectively. 

Hence $P_{j}(Z)=0,$ for $j=1,2,$
and non-zero for $j=0,3,4,5,6.$
We will combine Propositions \ref{gencontract} and \ref{charischrom}, and thus use the chromatic polynomial of each contracted graph to find the $P_j(Z)$. 

As is well known, the contraction of a cycle matroid of a graph at an edge set $S$ is the same as the same as the cycle matroid of the graph obtained by contracting this edge set $S$ (and also the same as the matroid obtained from any generator matrix of the shortening of any $C^q$ at the positions corresponding to the complement of $S$ in the entire edge set).

Hence we will contract in the graph sense  each  flat in question, and obtain a multigraph, and study the resulting graph obtained by simplification. 
We then obtain:
\begin{itemize}
\item The "empty" graph $K_1$, with just one vertex, and no edges, and with chromatic polynomial $Z$.
\item $4$ graphs of type $K_2,$ each with chromatic polynomial $Z(Z-1).$
\item $3$ graphs of type $K_2.$, each with chromatic polynomial $Z(Z-1).$
\item $6$ graphs of type $K_3$, each with chromatic polynomial $Z(Z-1)(Z-2).$
\item The complete graph $K_4$, with chromatic polynomial $Z(Z-1)(Z-2)(Z-3)$
\end{itemize}
Using Propositions \ref{gencontract} and \ref{charischrom}, one then gets, since each of the contracted graphs are connected:
\begin{align*}
P_{0}(Z) &= 1, \\
P_{3}(Z) &= 4(Z-1), \\
P_{4}(Z) &= 3(Z-1), \\
P_{5}(Z) &= 6(Z-1)(Z-2) = 6Z^2 - 18Z + 12, \\
P_{6}(Z) &= (Z-1)(Z-2)(Z-3) = Z^3 - 6Z^2 + 11Z - 6.
\end{align*}

In Subsection \ref{converting} we will show how one uses the formula 
$$P_{w}(q^m)=\sum_{r=1}^mA_w^{(r)}\prod_{i=0}^{r-1}(q^m-q^i) \textrm{ for }m> 0,$$ to determine  the $\{A^{(r)}_{w}\},$ for $w=1\cdots,6$ for each of the fixed values $r=1,2,3$ for all $C^q$.
}

\end{example}

\subsection{The lattice of flats} \label{flats}

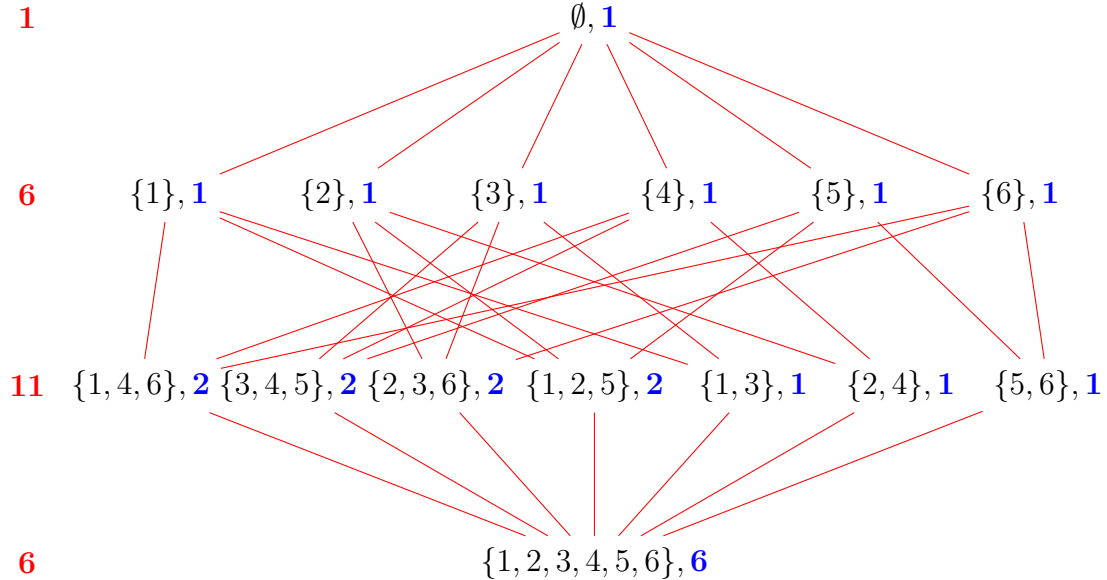
\begin{figure}[ht]
\centering
\begin{tikzpicture}[xscale=1.5, yscale=1.8]
  \node[red] at (-5, 2) {\textcolor{red}{\bfseries 1}};
  \node[red] at (-5, 0.7) {\textcolor{red}{\bfseries 6}};
  \node[red] at (-5, -0.7) {\textcolor{red}{\bfseries 11}};
  \node[red] at (-5, -2) {\textcolor{red}{\bfseries 6}};

  \node (max) at (0,2) {$\emptyset, \trygve{\bf{1}}$};

  \node (e1) at (-3.75, 0.7) {$\{1\}, \trygve{\bf{1}}$};
  \node (e2) at (-2.25, 0.7) {$\{2\}, \trygve{\bf{1}}$};
  \node (e3) at (-0.75, 0.7) {$\{3\}, \trygve{\bf{1}}$};
  \node (e4) at (0.75, 0.7)  {$\{4\}, \trygve{\bf{1}}$};
  \node (e5) at (2.25, 0.7)  {$\{5\}, \trygve{\bf{1}}$};
  \node (e6) at (3.75, 0.7)  {$\{6\}, \trygve{\bf{1}}$};

  \node (f146) at (-4, -0.7) {$\{1,4,6\}, \trygve{\bf{2}}$};
  \node (f345) at (-2.7, -0.7) {$\{3,4,5\}, \trygve{\bf{2}}$};
  \node (f236) at (-1.4, -0.7) {$\{2,3,6\}, \trygve{\bf{2}}$};
  \node (f125) at (0, -0.7)    {$\{1,2,5\}, \trygve{\bf{2}}$};
  \node (f13)  at (1.4, -0.7)  {$\{1,3\}, \trygve{\bf{1}}$};
  \node (f24)  at (2.7, -0.7)  {$\{2,4\}, \trygve{\bf{1}}$};
  \node (f56)  at (4, -0.7)    {$\{5,6\}, \trygve{\bf{1}}$};

  \node (min) at (0,-2) {$\{1,2,3,4,5,6\}, \trygve{\bf{6}}$};

  \foreach \i in {1,2,3,4,5,6} { \draw[red] (max) -- (e\i); }

  \draw[red] (e1) -- (f146); \draw[red] (e4) -- (f146); \draw[red] (e6) -- (f146);
  \draw[red] (e3) -- (f345); \draw[red] (e4) -- (f345); \draw[red] (e5) -- (f345);
  \draw[red] (e2) -- (f236); \draw[red] (e3) -- (f236); \draw[red] (e6) -- (f236);
  \draw[red] (e1) -- (f125); \draw[red] (e2) -- (f125); \draw[red] (e5) -- (f125);
  \draw[red] (e1) -- (f13);  \draw[red] (e3) -- (f13);
  \draw[red] (e2) -- (f24);  \draw[red] (e4) -- (f24);
  \draw[red] (e5) -- (f56);  \draw[red] (e6) -- (f56);

  \foreach \f in {f146, f345, f236, f125, f13, f24, f56} { \draw[red] (\f) -- (min); }

\end{tikzpicture}
\caption{Lattice of flats for the cycle matroid of $K_4$. The blue numbers indicate $|\mu(\emptyset, F)|$ and the red numbers indicate Whitney numbers of the first kind.}
\label{figure2}
\end{figure}

We now return to Example \ref{contractionexample}, and the codes and matroid obtained from the complete graph $K_4:$

In Figure~\ref{figure2}
we have drawn (upside down) the lattice of flats, ordered by inclusion. Reading upwards and taking complements, one obtains the poset of cycles of $M^*$. 

The absolute values of the local Whitney numbers $\mu(\emptyset,F)$ of the  flats are written in blue to the right of each node, while knowing that these Whitney numbers themselves alternate in sign for each immediate row change. The absolute values of the Whitney numbers of the first kind are written in red to the left of each row.  By Theorem \ref{thm:Mobius} we deduce from these numbers at the left that $P_{6}(Z)=Z^3-6Z^2+11Z-6.$ A lower left part of the main diagram is drawn in Figure~\ref{figure3},

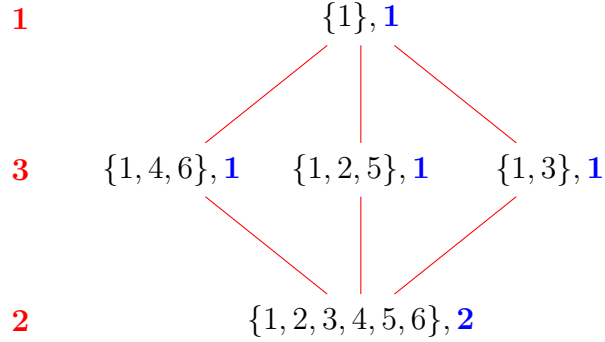
\begin{figure}[ht]
\centering
\begin{tikzpicture}[node distance=2cm, auto]
\node at (-4.5, 0) {\textcolor{red}{\bfseries 1}};
\node at (-4.5, -2) {\textcolor{red}{\bfseries 3}};
\node at (-4.5, -4) {\textcolor{red}{\bfseries 2}};

  \node (top) at (0,0) {$\{1\}, \trygve{\bf{1}}$};

  \node (f146) at (-2.5, -2) {$\{1,4,6\}, \trygve{\bf{1}}$};
  \node (f125) at (0, -2)    {$\{1,2,5\}, \trygve{\bf{1}}$};
  \node (f13)  at (2.5, -2)  {$\{1,3\}, \trygve{\bf{1}}$};

  \node (bottom) at (0,-4) {$\{1,2,3,4,5,6\}, \trygve{\bf{2}}$};

  \foreach \n in {f146, f125, f13} {
    \draw [red] (top) -- (\n);
    \draw [red] (\n) -- (bottom);
  }
\end{tikzpicture}
\caption{The absolute values of Whitney numbers for the contracted matroid $M/\{1\}$.}
\label{figure3}
\end{figure}

Looking at this diagram and the five other analogous diagrams with singleton flats in the top row, we deducs that  $P_{5}(Z)=6(Z^2-3Z+2)=6Z^2-18Z+12,$
 using Theorem \ref{thm:Mobius} on each of the $6$ contracted matroids, and summing contributions. The red and blue numbers are now absolute values of Whitney numbers for the contracted matroid $M/\{1\}.$  
 The even lower left corner, shown as Figure~\ref{figure4}, 
\begin{figure}[ht]
\centering
\begin{tikzpicture}[node distance=2cm, auto]
\node[color=red] at (-4.5, 0)  {\textcolor{red}{\bfseries 1}};
\node[color=red] at (-4.5, -2) {\textcolor{red}{\bfseries 1}};

  \node (top) at (0,0) {$\{1, 4, 6\}, \trygve{\bf{1}}$};

  \node (bottom) at (0,-2) {$\{1, 2, 3, 4, 5, 6\}, \trygve{\bf{1}}$};

  \draw [red] (top) -- (bottom);
\end{tikzpicture}
\caption{The absolute values of Whitney numbers for the contracted matroid $M/\{1, 4, 6\}$.}
\label{figure4}
\end{figure}
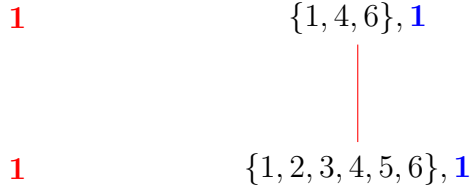
and the three analogous diagrams with triplets in the top row,  give $P_{3}(Z)=4(Z-1)=4Z-4$
summing contributions and using Theorem \ref{thm:Mobius}.  The red and blue numbers are now absolute values of Whitney numbers for the contracted matroid $M/\{1,4,6\}. $ The lower and rightmost corner is shown in Figure~\ref{figure5}.
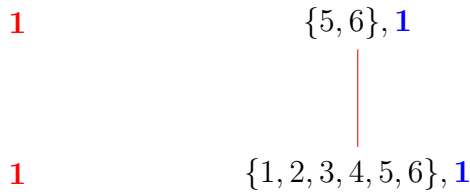
\begin{figure}[ht]
\centering
\begin{tikzpicture}[node distance=2cm, auto]
  \node[red] at (-4.5, 0)  {\textcolor{red}{\bfseries 1}};
  \node[red] at (-4.5, -2)  {\textcolor{red}{\bfseries 1}};

  \node (top) at (0,0) {$\{5, 6\}, \trygve{\bf{1}}$};

  \node (bottom) at (0,-2) {$\{1, 2, 3, 4, 5, 6\}, \trygve{\bf{1}}$};

  \draw [red] (top) -- (bottom);
\end{tikzpicture}
\caption{The absolute values of Whitney numbers for the contracted matroid $M/\{5,6\}$.}
\label{figure5}
\end{figure}

That diagram and the two analogous diagrams with point pairs in the top row,  give $P_{4}=3(Z-1)=3Z-3$
 using Theorem \ref{thm:Mobius}, and summing contributions. The red and blue numbers are now absolute values of  Whitney numbers for the contracted matroid $M/\{5,6\}.$ 

 We have now reproduced the results from Subsection \ref{square}, where we used chromatic polynomials, using Theorem \ref{thm:Mobius}
(Formula (\ref{eq:Mobius-flats})) instead. This example is of course very simple, but illustrates how one can find weight polynomials of a linear code $C$, in general, using the lattice of flats of $M$.

\subsection{The method of elongations}
In this subsection we will show how the method of elongations described in Subsection \ref{eq:Mobius-flats} can be used to find the 
$P_j(Z)$ for the cycle matroid of the same complete graph $K_4.$
In Figure~\ref{figure6} we have drawn the cycle diagram obtained from Example \ref{contractionexample}, for the case $N=M^*$, for the matroid $M$ appearing there, coming from the graph $K_4.$
\begin{figure}[ht]
\centering
\begin{tikzpicture}[xscale=1.2, yscale=1.5]
  \node (max) at (0,3) {$E, \trygve{\bf{6}}$};

  \node (c1) at (-5,1) {$E \setminus \{1\}, \trygve{\bf{2}}$};
  \node (c2) at (-3,1) {$E \setminus \{2\}, \trygve{\bf{2}}$};
  \node (c3) at (-1,1) {$E \setminus \{3\}, \trygve{\bf{2}}$};
  \node (c4) at (1,1)  {$E \setminus \{4\}, \trygve{\bf{2}}$};
  \node (c5) at (3,1)  {$E \setminus \{5\}, \trygve{\bf{2}}$};
  \node (c6) at (5,1)  {$E \setminus \{6\}, \trygve{\bf{2}}$};

  \node (s235) at (-4.5,-1) {$\{2,3,5\}, \trygve{\bf{1}}$};
  \node (s126) at (-3,-1)   {$\{1,2,6\}, \trygve{\bf{1}}$};
  \node (s145) at (-1.5,-1) {$\{1,4,5\}, \trygve{\bf{1}}$};
  \node (s346) at (0,-1)    {$\{3,4,6\}, \trygve{\bf{1}}$};
  \node (s2456) at (1.5,-1) {$\{2,4,5,6\}, \trygve{\bf{1}}$};
  \node (s1356) at (3,-1)   {$\{1,3,5,6\}, \trygve{\bf{1}}$};
  \node (s1234) at (4.5,-1) {$\{1,2,3,4\}, \trygve{\bf{1}}$};

  \node (min) at (0,-3) {$\emptyset, \trygve{\bf{1}}$};

  \foreach \i in {1,2,3,4,5,6} { \draw [red] (max) -- (c\i); }

  \draw [red] (c1) -- (s126); \draw [red] (c1) -- (s145); \draw [red] (c1) -- (s1356); \draw [red] (c1) -- (s1234);
  \draw [red] (c2) -- (s235); \draw [red] (c2) -- (s126); \draw [red] (c2) -- (s2456); \draw [red] (c2) -- (s1234);
  \draw [red] (c3) -- (s235); \draw [red] (c3) -- (s346); \draw [red] (c3) -- (s1356); \draw [red] (c3) -- (s1234);
  \draw [red] (c4) -- (s145); \draw [red] (c4) -- (s346); \draw [red] (c4) -- (s2456); \draw [red] (c4) -- (s1234);
  \draw [red] (c5) -- (s235); \draw [red] (c5) -- (s145); \draw [red] (c5) -- (s2456); \draw [red] (c5) -- (s1356);
  \draw [red] (c6) -- (s126); \draw [red] (c6) -- (s346); \draw [red] (c6) -- (s2456); \draw [red] (c6) -- (s1356);

  \foreach \node in {s235, s126, s145, s346, s2456, s1356, s1234} { \draw [red] (\node) -- (min); }

\end{tikzpicture}
\caption{Cycle diagram for the dual matroid $N=M^*$ of the graph $K_4$.}
\label{figure6}
\end{figure}
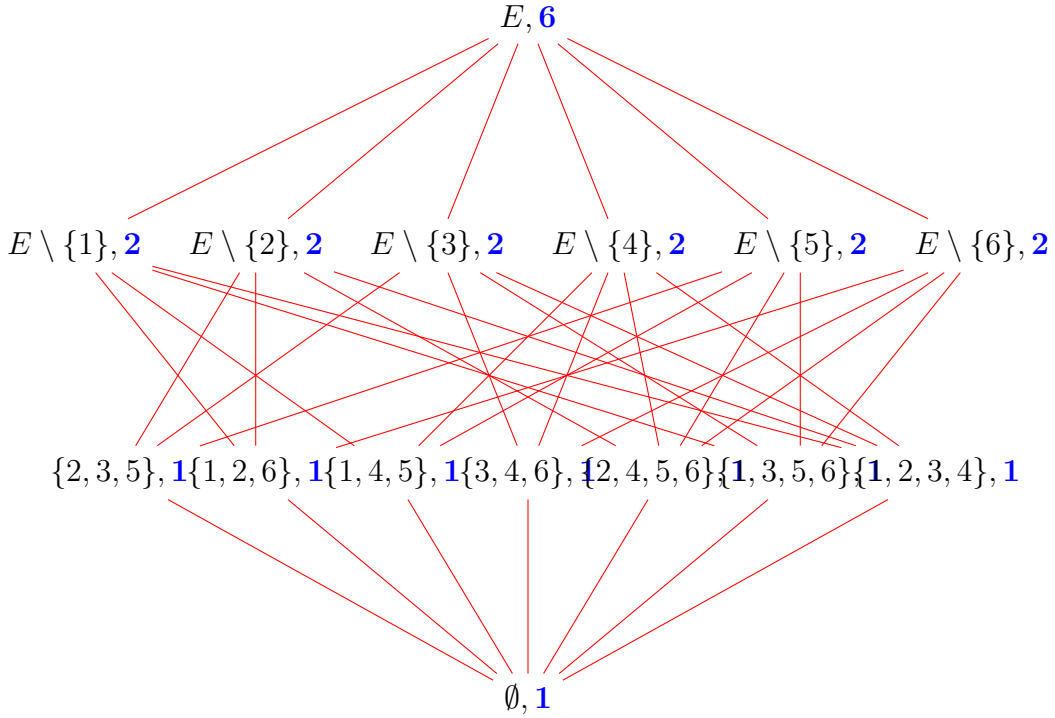

 The numbers in blue to the right of each node are the $|\mu(\emptyset,\sigma)|$ for the Möbius function of the poset. Let us call such a number $\nu_{i,\sigma}=\nu^{(0)}_{i,\sigma}$ if $\sigma$ is a cycle of nullity $i$. The corresponding diagram for the first elongation of $N(=M^*)$is drawn in Figure~\ref{figure7}.
\begin{figure}[ht]
\centering
\begin{tikzpicture}[node distance=1.5cm, auto]
  \node (max) at (0,2.5) {$E, \trygve{\bf{5}}$};

  \node (c1) at (-5,0) {$E \setminus \{1\}, \trygve{\bf{1}}$};
  \node (c2) at (-3,0) {$E \setminus \{2\}, \trygve{\bf{1}}$};
  \node (c3) at (-1,0) {$E \setminus \{3\}, \trygve{\bf{1}}$};
  \node (c4) at (1,0)  {$E \setminus \{4\}, \trygve{\bf{1}}$};
  \node (c5) at (3,0)  {$E \setminus \{5\}, \trygve{\bf{1}}$};
  \node (c6) at (5,0)  {$E \setminus \{6\}, \trygve{\bf{1}}$};

  \node (min) at (0,-2.5) {$\emptyset, \trygve{\bf{1}}$};

  \foreach \n in {c1, c2, c3, c4, c5, c6} {
    \draw [red] (min) -- (\n);
    \draw [red] (max) -- (\n);
  }
\end{tikzpicture}
\caption{Cycle diagram for the first elongation of $N=M^*$ for $K_4$. The blue numbers represent the virtual Betti numbers $\nu^{(1)}_{i,\sigma}$.}
\label{figure7}
\end{figure}
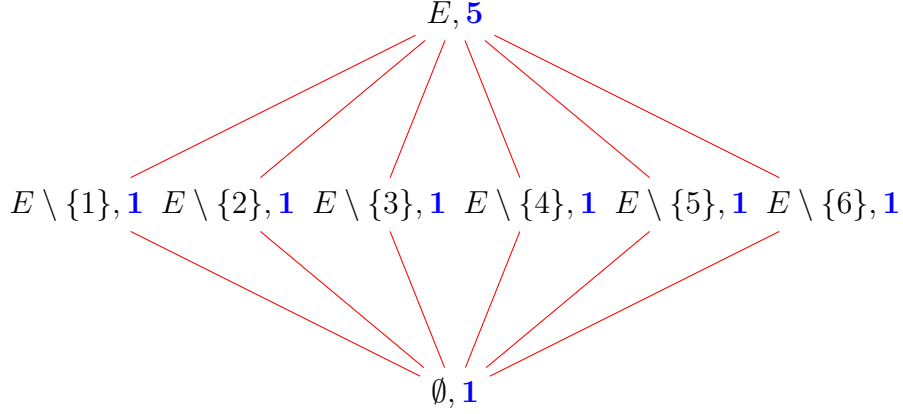

The analogous numbers in blue are now called  $\nu^{(1)}_{i,\sigma}$ where $i=n^{(1)}(\sigma)$ for each cycle in question. 
The last diagrams, for the second anf third elongations,  are drawn in Figure~\ref{figure8}.
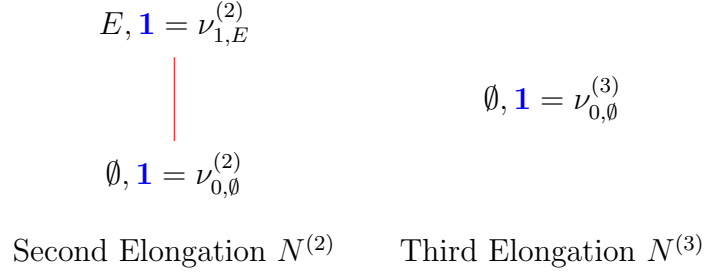
\begin{figure}[ht]
\centering
\begin{tikzpicture}[node distance=2cm]
  \node (max2) at (0,2) {$E, \trygve{\bf{1}} = \nu^{(2)}_{1,E}$};
  \node (min2) at (0,0) {$\emptyset, \trygve{\bf{1}} = \nu^{(2)}_{0,\emptyset}$};
  \draw [red] (min2) -- (max2);
  \node at (0,-1) {{Second Elongation} $N^{(2)}$};

  \node (min3) at (5,1) {$\emptyset, \trygve{\bf{1}} = \nu^{(3)}_{0,\emptyset}$};
  \node at (5,-1) {{Third Elongation} $N^{(3)}$};
\end{tikzpicture}
\caption{Cycle diagrams for the second and third elongations of $N=M^*$ for $K_4$.}
\label{figure8}
\end{figure}

Set $\nu^{(l)}_{i,j}=\Sigma \nu^{(l)}_{i,\sigma}$, where the sum is taken over the $\sigma$ of cardinality $j$ (and nullity $i$ for $N^{(l)}$). From \cite{JohnsenRoksvoldVerdure} it now follows that $$P_j(Z)=\sum_{\ell=0}^{3}(\phi_j^{(l)}-\phi_j^{(l-1)})Z^l,$$
for $i=0,\cdots,6,$ where $$\phi_j^{(l)}
=\displaystyle\sum_{i=0}^k(-1)^i\nu^{(l)}_{i,j} \textrm{,  for } j=0, 1, \dots, n. 
$$
From this formula and the diagrams above we obtain the following non-zero $\phi^{(l)}_{j}$ in our example: 
$$\phi_0^{(0)}=1,\phi_3^{(0)}=-4, \phi_4^{(0)}=-3, \phi_5^{(0)}=12, \phi_6^{(0)}=-6,$$
$$ \phi_0^{(1)}=1, \phi_5^{(1)}=-6, \phi_6^{(1)}=5, \phi_0^{(2)}=1,\phi_6^{(2)}=-1, \phi_0^{(3)}=1.$$
Formula (\ref{weightelong})  gives the expected polynomials: 
 $$P_{0}=1, P_{3}=4Z-4, P_{4}=3Z-3$$
 $$P_{5}=6Z^2-18Z+12,P_{M^*,6}=Z^3-6Z^2+11Z-6.$$

\subsection{Using the OS-algebra, method with NBC-sets} \label{Using the OS-algebra, method with NBC-sets}
To construct the No Broken Circuit (NBC) sets for the cycle matroid of the complete graph $K_4$, we choose a total ordering of the edges, find all cycles of the graph (circuits of the matroid), form the broken circuits by deleting the minimum edge from each cycle, and collect all edge subsets that do not contain any broken circuit.

From the drawing in Subsection \ref{square}, the circuits of the matroid are:
\begin{equation} 
\mathcal{C} = \{ \{1,4,6\}, \{3,4,5\}, \{1,2,5\}, \{2,3,6\}, \{1,2,3,4\}, \{1,3,5,6\}, \{2,4,5,6\} \}. 
\end{equation}
Choosing the natural ordering $1 < 2 < 3 < 4 < 5 < 6$, the broken circuits are:
\begin{equation} 
\{4,6\}, \{4,5\}, \{2,5\}, \{3,6\}, \{2,3,4\}, \{3,5,6\}, \{4,5,6\}. 
\end{equation}

Hence, the NBC-sets (subsets containing no broken circuits) are:
\begin{itemize}
    \item The empty set $\emptyset$ (1 set);
    \item All $6$ singleton sets $\{i\}$ (6 sets);
    \item All $15$ two-point sets except $\{4,6\}, \{4,5\}, \{2,5\},$ and $\{3,6\}$ (11 sets);
    \item The $6$ three-point sets $\{1,2,3\}, \{1,2,4\}, \{1,2,6\}, \{1,3,4\}, \{1,3,5\},$ and $\{1,5,6\}$.
\end{itemize}
This results in $1$ zero-point set, $6$ one-point sets, $11$ two-point sets, and $6$ three-point sets. Consequently, the Poincaré series $P_{OS(Z)}$ of the Orlik–Solomon algebra is $1 + 6Z + 11Z^2 + 6Z^3$.

Hence for the "highest" generalized weight polynomial we have:
$$P_6(Z)=p(M,Z)=Z^3P_{\mathcal{A}}(\frac{-1}{Z})=Z^3-6Z^2+11Z-6.$$
For $P_5(Z)$ we have to add the contributions from $6$ contracted graphs of type $K_3.$
For each of them we have one circuit 123, and one broken circuit 23. Hence the NBC-sets are 
$\emptyset,$ all the $3$ singleton sets, and the two-point sets 12 and 13.
Hence the Poincare series $P_{OS}(Z)$ of the OS algebra of the relevant matroid is 
$1+3Z+2Z^2$
Hence we have:
$$P_5(Z)=6Z^2P_{OS}(\frac{-1}{Z})=6(Z^2-3Z+2).$$
For the  flats contributing to $P_4(Z)$ and $P_3(Z)$ we look at $3$ and $4$, respectively, graphs $K_2$ with no circuits at all, and therefore no broken circuits. Hence for each of them the NBC-sets are  $\emptyset,$ and a singleton set, so each Poincare series is $1+Z$. Then the contributions to the generalized weight polynomials are $Z-1$ for each of them.  This gives $P_4(Z)=3(Z-1),$ and $P_3(Z)=4(Z-1).$

For $P_0(Z)$ we contract all edges, so the ground set of the resulting matroid is $\emptyset$.
This is also the only NBC-set. Hence both the Poincare polynomial of the resulting(trivial)  OS-albebra and the generalized weight polynomial is $P_0(Z)=1.$

\subsection{Using order homology} \label{usingorder}

We will now utilize the viewpoint of Section \ref{homol}, and in particular Remark \ref{homolo}.
We will sketch how you obtain some of the coefficients of the generalized weight polynomials.
We start with the coefficient $-6$ in the constant term $-6$ of $P_6(Z)=Z^3-6Z^2+11Z-6.$
This corresponds to the highest, cubic term of the Poincare series $6Z^3+11Z^2+6Z+1.$ 
That the coefficient is negative is due to the fact that the rank of the matroid is odd. So we will show that the absolute value is $6.$ But from Remark \ref{homolo} we know that this absolute value is that of the reduced Euler characteristic of the simplicial complex $S_L$ described in the beginning of Subsection \ref{homol}. Looking at Figure~\ref{figure2}, we see that the facets have cardinality $2$, and the face number $f_2$ is the number of arrows connecting the middle layers of the drawing in that figure.
As one sees, there are $18$ arrows. The face number $f_1$ is the number of lattice points after deleting the top and bottom. There are $13$ such points. The face number $f_0$ (contributing because we look at $\underline{reduced}$ homology) is $1$ because of the empty set. Hence the absolute value of the reduced Euler characteristic is $18-13+1=6$. 

To find the other coefficients of $P_6(Z)$ we add contributions from all other flats, For a fixed flat in the lattice diagram we then only look at the subdiagram of those flats that are at or above that fixed flat in the diagram. Then the coefficient $11$ of $P_6(Z)$ arises as $4(3-1)+3(2-1)$, since the face numbers for the $4$ flats with $3$ elements are $f_1=3,f_0=1$, and for the $3$ flats of cardinality  of cardinality $2$ they are $f_1=2,f_0=1.$ For the coefficient $-6$ of the linear term of $P_6(Z)$, we look at 6 diagrams, which only have a top and a bottom points. Removing both, we get 6 empty diagrams, each with $f_0=1.$  

To find $P_5(Z)$ in a similar way, we must look at all $6$ flats of cardinality $n-5=6-5=1$, i.e. the ones just below the top. 
For each of them we now instead look only at the flats at or under that flat in the diagram.  Restricting calculations to subdiagrams bounded above by a flat $F$  mirrors the algebraic behavior of the contraction matroid $M/F$. To find the contribution to the constant term of $P_5(Z)$, we delete the top and bottom term of the new, smaller diagram obtained. That diagram is given in Figure~\ref{figure3}. We end up with just $3$ points. Hence we get face numbers $f_1=3,f_0=1.$ This gives reduced Euler characteristic $2$, and the contribution 
$12$ for the sum of the $6$ contributions. This explains the coefficient $12$, both  in $P_5(Z)=6Z^2-18Z+12, $
and in $12Z^2+18Z+6,$ which is the sum of the Poincare series of the OS-algebras, of all $6$ matroids contracted at one point(edge). 

All other coefficients of all weight polynomials can be found by looking at other suitable subdiagramss of Figure~\ref{figure2}.


\subsection{The method of Jurrius and Pellikaan} \label{JPstuff}
 It is in general  well known that the entire weight enumerator $W(X,Y,Z)=\Sigma_{w=0}^n P_w(Z)X^{n-w}Y^w$ can be found from the Tutte polynomial: If one knows the Tutte polynomial, all generalized weight polynomials $P_w(Z)$ can be extracted from the following formula:

$$W(X,Y,Z)=(X-Y)^{r(M)}Y^{n-r(M)}T_M(\frac{X+(Z-1)Y)}{(X-Y)},\frac{X}{Y}).$$

So  all the techniques in the previous sections are primarily intended to be used in the case when one does not know the Tutte polynomial, but nevertheless has sufficient knowledge about the lattice and cardinalities of flats. To know the Tutte polynomial is equivalent to know all the numbers
$B_{j,s}$  that are the number of subsets  $J \subset E$ such that $|J|=j$ and $r(J)=s$, where $r$ is the rank function of the matroid $M$ obtained from a generator matrix of the linear $[n,k,d]_q-$code $C$. 
We will present  another "twist" on this, and present a technique from \cite{JurriusPellikaanCodes} that does not use the lattice of flats of the matroid, but uses knowledge about \underline{all} subsets of the ground set:
Set 
	\begin{equation} \label{JPsum1}
    B_i (T)=\Sigma_{s=0}^{k-1}B_{i,s} (T^{k-s}-1).
    \end{equation}
We have $P_0 (Z)=1$, and for $j>0$:

\begin{equation} \label{JPsum2}
P_j (Z)=\Sigma_{i=n-j}^{n-d}(-1)^{n+j+i} { \binom{i}{n-j}} B_i (T).
\end{equation}

Hence a practical way to find all the weight polynomials $P_j(Z)$ may be to find all the 
$B_{i,s}$ in question. 

\begin{remark} \label{more}
{\rm Using this terminology for the matroid from $K_4$ in  Example \ref{contractionexample} we get the values $B_{3,2}=4, B_{2,2}=15, B_{1,1}=6,B_{0,0}=1$ for the non-zero contributions that matter in Formula (\ref{JPsum1}). This gives 
($B_6(Z)=B_5(Z)=B_4(Z)=0,$ and) $B_3(Z)=4(Z-1), B_2(Z)=15(Z-1), B_1(Z)=6(Z^2-1), B_0(Z)=Z^3-1).$
Thus Formula (\ref{JPsum2}) gives:}

$\begin{array}{rcl}
P_6(Z)&=&B_0(Z)-B_1(Z)+B_2(Z)-B_3(Z)=Z^3-6Z^2+11Z-6.\\
P_5(Z)&=&B_1(Z)-2B_2(Z)+3B_3(Z)=6Z^2-18Z+12.\\
P_4(Z)&=&B_2(Z)-3B_3(Z)=3(Z-1)\\
P_3(Z)&=&B_3(Z)=4(Z-1).
\end{array}$

and clearly $P_0(Z)=1.$
We see that we (of course) get the same answer as in Example \ref{contractionexample}.
\end{remark}

\begin{remark}
A method of finding all the generalized weight polynomials, after first finding all $B_{j,s}$, was used by Kaipa and Pradhan in \cite{KaipaPradhan} to find all the $P_j(Z)$ for the projective Reed-Muller code $PR_3(2,3).$

\end{remark}

\subsection{Finding the weight spectra from the generalized weight polynomials} \label{converting}

We will illustrate how one finds the weight spectra from the generalized weight polynomials.
A weight spectrum is a set of numbers $A^{(r)}_w$ for fixed $r$, where $w$ varies. The conversion formula Equation (\ref{conversion}) 
$$P_w(q^m)=\sum_{r=0}^mA_w^{(r)}\prod_{i=0}^{r-1}(q^m-q^i), \textrm{for }m>0,$$
gives us all $A^{(r)}_w$ for fixed $w$, where $r$ varies. But if one uses the conversion formula for all fixed $w$, one ends up with all $A^{(r)}_w$ anyway. In our running example with the cycle matroid of $K_4,$
let us use Equation (\ref{conversion}) to find the $A^{r}_w$ for $w=6$. It is enough to use the formula for $m=1,2,3=r(M).$ Obviously $A_6^{(0)}=0$.

For $m=1$ we get $q^3-6q^2+11q-6=A^{(1)}_6(q-1).$ This gives $A^{(1)}_6=q^2-5q+6.$

For $m=2$ we get $q^6-6q^4+11q^2-6=A^{(1)}_6(q^2-1)+A^{(2)}_6(q^2-1)(q^2-q).$
This gives $A^{(2)}_6=q^2+q-5,$ after inserting the value found for $A^{(1)}_6.$

For $m=3$ we get $q^9-6q^6+11q^3-6=A^{(1)}_6(q^3-1)+A^{(2)}_6(q^3-1)(q^3-q)+A^{(3)}_6(q^3-1)(q^3-q)(q^3-q^2).$ This gives (the obvious answer) $A_6^{(3)}=1,$ after inserting the values found for $A^{(1)}_6$ and $A^{(2)}_6.$

One finds the $A^{(r)}_w$ for $w=0,3,4,5$ in the same manner, and  the non-zero values are:
$A^{(1)}_5=6q-12, A^{(2)}_5=6, A^{(1)}_4=3,A^{(1)}_3=4, A^{(0)}_0=1.$
For larger examples the task of finding the $A^{(r)}_w$ from the weight polynomials is most easily done using a computer program like SAGE.

\section{Cyclic flats, cloud/flock polynomials, and top weights}

The Tutte polynomial of a matroid, and hence the top weight polynomial, can be reconstructed from suitable cyclic-flat data through Eberhardt' cloud/flock polynomials~\cite{Eberhardt}. Evaluating those polynomials generally encodes information about flats and subsets lying over each cyclic flat.

\begin{definition}[Cyclic flats]\label{def:cyclic-flats}
A \emph{cyclic flat} of a matroid $M$ is a subset $X\subseteq E$ which is both a flat and a cycle (i.e.\ inclusion-minimal among subsets with the same nullity). Equivalently, its complement in $E$ is a cyclic flat in the dual matroid $M^*$.

Let $L(M)$ be the lattice of flats and $Z(M)$ the set of cyclic flats.

Define the map
\begin{equation*}
e_M : L(M) \to Z(M),
\end{equation*}
sending a flat $X$ to the unique maximal cyclic flat contained in $X$.

Let
\begin{equation*}
\operatorname{cl}_M : \mathcal{P}(E)\to L(M)
\end{equation*}
be the closure operator ($cl_M(A)$ is the smallest flat containing $A$, that is: The intersection of all flats containing $A$).
\end{definition}

\begin{definition}[Cloud and flock polynomials {\cite{Eberhardt}}]\label{def:cloud-flock}
For each cyclic flat $Z\in Z(M)$:
\begin{itemize}
  \item the \emph{cloud polynomial} is
  \begin{equation*}
  c(M,Z;X)
  =
  \sum_{A\in e_M^{-1}(\{Z\})} X^{r(M)-r(A)},
  \end{equation*}
  \item the \emph{flock polynomial} is
  \begin{equation*}
  f(M,Z;Y)
  =
  \sum_{B\in \operatorname{cl}_M^{-1}(\{Z\})} Y^{|B|-r(B)}.
  \end{equation*}
\end{itemize}
\end{definition}
The following is known:
\begin{theorem}[Cloud/flock decomposition of Tutte {\cite{Eberhardt}}]\label{thm:Tutte-cloud-flock}
For any matroid $M$, the Tutte polynomial can be computed from its cyclic flats via
\begin{equation*}
T_M(X,Y)
=
\sum_{Z\in Z(M)}
c(M,Z;X-1)\, f(M,Z;Y-1).
\end{equation*}
\end{theorem}

We now use:
\begin{equation*}
P_{n}(\lambda)
=
p(M,\lambda)
=
(-1)^{r(M)} T_M(1-\lambda,0),
\end{equation*} and obtain a cyclic-flat expansion:
\begin{corollary}[Top weight polynomial via cyclic flats]\label{cor:top-cyclic}

\begin{equation*}
P_{n}(\lambda)
=
(-1)^{r(M)}
\sum_{Z\in Z(M)}
c(M,Z; -\lambda)\, f(M,Z; -1).
\end{equation*}
\end{corollary}

\begin{proof}
Substitute $X=1-\lambda$ and $Y=0$ into the cloud--flock decomposition of $T_M(X,Y)$. 
\end{proof}

\begin{example}[Uniform matroid]\label{ex:uniform-cyclic}
Let $M = U(r,n)$ be the uniform matroid with $0 <r<n.$ Then
\begin{equation*}
Z(M) = \{\emptyset, E\}.
\end{equation*}
As shown in \cite{Eberhardt},
\begin{equation*}
c(M,E;X) = 1, \qquad f(M,\emptyset;Y) = 1,
\end{equation*}
and
\begin{equation*}
c(M,\emptyset;X) = \sum_{i=0}^{r-1} \binom{n}{i} X^{r-i}, 
\qquad
f(M,E;Y) = \sum_{i=r}^{n} \binom{n}{i} Y^{i-r}.
\end{equation*}
Substituting into the cyclic-flat expression for $P_{n}(\lambda)$ gives
\begin{equation*}
P_{n}(\lambda)
=
(-1)^r
\left(
\sum_{i=0}^{r-1} \binom{n}{i} (-\lambda)^{\,r-i}
+
\sum_{i=r}^{n} \binom{n}{i} (-1)^{\,i-r}
\right).
\end{equation*}
Pulling the signs inside and simplifying yields
\begin{equation*}
P_{n}(\lambda)
=
\sum_{i=0}^{r-1} \binom{n}{i} (-1)^i \lambda^{r-i}
+
\sum_{i=r}^{n} \binom{n}{i} (-1)^i,
\end{equation*}
which agrees with the direct Möbius computation in Remark~\ref{ex:uniform-top} below.

To find the $P_j(\lambda)$ for all $j<n$, one uses Proposition \ref{gencontract}  (summing over all contractions of flats of cardinality $n-j$), Example \ref{unifo}, (4) and (7), which tell which sets are flats, and what the contractions at these flats are, and the calculationns of Example \ref{ex:uniform-cyclic} for each of these contractions.
\end{example}


\begin{remark}[Uniform matroids and Möbius]\label{ex:uniform-top}
As a sanity check we use the part of Definition \ref{char} to compute $P_n(\lambda)=p(M,\lambda)$ directly:

The flats are $E$ and all subsets of $E$ of cardinality $<r$. One checks 
that:
\begin{itemize}
  \item For a flat $F$ with $|F|=i \le r-1$, $\mu(\emptyset,F) = (-1)^{i}$.
  \item For $F=E$, one has
  \begin{equation*}
  \mu(\emptyset,E)
  =
  \sum_{0\le i<r} \binom{n}{i} (-1)^{i-1}
  =
  \sum_{r\le i\le n} \binom{n}{i} (-1)^i.
  \end{equation*}
\end{itemize}
Thus,
\begin{equation}\label{eq:uniform-top-formula}
P_{n}(\lambda)
= 
p(M,\lambda)
=
\sum_{0\le i<r} \binom{n}{i} (-1)^i \lambda^{r-i}
+
\sum_{r\le i\le n} \binom{n}{i} (-1)^i.
\end{equation}
This matches the computation from Example~\ref{ex:uniform-cyclic} via cyclic flats.
\end{remark}
\begin{remark}
From a superficial viewpoint the method of this subsection uses even less information about the matroid than those methods that use the entire lattice, and cardinalities, of the flats, since we only use the cyclic flats. 
But in reality this is a hybrid method in comparison with methods where we only use the flats, like in Subsection \ref{flats}, and methods where we use information about all subsets of the ground set $E$, like in Section \ref{JPstuff}. This is because the map $e_M$ is defined only for the flats, while the map $cl_M$ is defined for  \underline{all} subsets of $E$.
\end{remark}

\section{Projective Reed-Muller codes, an example}

Reed–Muller codes can be defined over any finite field $\mathbb{F}_q$, where $q$ is a prime power. Let $m$ and $d$ be positive integers, with $m > d$. A message $x \in \mathbb{F}_q^k$, where $k = \binom{m+d}{m}$, is identified with an $m$-variate polynomial $p_x$ of total degree at most $d$. Since such a polynomial has exactly $\binom{m+d}{m}$ coefficients, this correspondence is one-to-one. The Reed–Muller encoding of $x$ is obtained by evaluating $p_x$ at every point of the affine space $\mathbb{F}_q^m$, yielding a code of block length $n = q^m$.

A natural projective analogue is obtained by replacing the affine space with a projective space, leading to the family of projective Reed–Muller codes. These codes admit a geometric interpretation through projective varieties and embeddings, such as the Veronese map. Below, we examine the projective Reed–Muller code $PRM_3(2,2)$, the $[13,6,6]_3$-code obtained from mapping $\mathbb{P}^2$
over $\mathbb{Z}_3$ into $\mathbb{P}^5$ over $\mathbb{Z}_3$ by the quadratic Veronese mapping. Details about such codes $C_q$, not only over $\mathbb{Z}_3$, but over any finite field $\mathbb{F}_q$, are given in \cite{JohnsenPratiharVerdureVeronese}.
In that paper, the generalized weight spectra were derived using Betti numbers from elongation matroids to first determine the generalized weight polynomials. The following result was used, but not published, in \cite{JohnsenPratiharVerdureVeronese}:

\begin{proposition} \label{liste3}
Let the Veronese code $C_q$ be the linear $[q^2+q+1,6]_q $-code with generator matrix $G_q$. Then for the Veronese code $C_3$ the nonzero generalized weight polynomials for the Veronese code $C_3$ are as follows

\begin{align*}
P_{13}(Z) &= Z^6 - 13Z^5 + 78Z^4 - 273Z^3 + 585Z^2 - 702Z + 324 \\
P_{12}(Z) &= 13Z^5 - 156Z^4 + 806Z^3 - 2223Z^2 + 3120Z - 1560 \\
P_{11}(Z) &= 78Z^4 - 780Z^3 + 3042Z^2 - 5148Z + 2808 \\
P_{10}(Z) &= 234Z^3 - 1638Z^2 + 3510Z - 2106 \\
P_{9}(Z) &= 13Z^3 - 117Z^2 + 234Z - 130\\
P_8(Z) &= 117Z^2 - 468Z + 351 \\
P_6(Z) &= 78Z - 78 \\
P_0(Z) &= 1
\end{align*}

while $P_1(Z),\cdots,P_5(Z)$ and $P_7(Z)$ all are identically zero.

\end{proposition}
We would like to show how one can find $P_{13}(Z)(=p(M,Z))$ for $M$ a generator matrix of the code, using 
  \begin{equation*}
    p(L, \lambda) = \sum_{x \in L} \mu(0, x)\,\lambda^{r - r(x)}.
    \end{equation*}
    from Definition \ref{char} instead, and thereafter, at least in part,  by the method of NBC-sets, described in  Subsection \ref{Using the OS-algebra, method with NBC-sets}. The weight polynomials $P_{13}(Z$ for $j \le 12$ are found by using the same techniques for matroids that are contractions of $M$ in its flats. These flats will be identified while finding
    $P_{13}(Z).$ By the proof of \cite[Theorem 21]{JohnsenPratiharVerdureVeronese} the flats of $M$ are:
    
    The empty set of rank $0$, and $13$ singleton sets of rank $1$, and $78$ point-pairs of rank $2$, and $13$ lines of rank $3$ and cardinality $4$, and $234$ "triangles" of rank and cardinality  $3$, and $234$ "quadrilaterals" of rank and cardinality $4$, and $117$ sets of type a line + a point outside the line. These have rank $4$ and cardinality $5$.
    Furthermore we have $78$ line pairs, each of rank $5$ and cardinality $7$, and at last we have the ground set $E=\mathbb{P}^2$ over $\mathbb{Z}_3$, with rank $6$ and cardinality $13$. Employing Definition \ref{Mob} one obtains that $\mu(\emptyset,F)$ have the following values when $F$ is:
   $$\emptyset\textrm{      1}$$
    $$\textrm{singleton  -1}$$
    $$\textrm{point pair   1}$$
    $$\textrm{line   -3  and triangle   -1}$$
    $$\textrm{quadrilateral   1  and line+point 3}$$
    $$\textrm{line pair  -9}$$
    $$\textrm{entire plane  324}$$
    Here each line corresponds to each rank, from $0$ to $6$, and summing over all sets in each line, we obtain 
    the values $$1,-13,78,-273, 585,-702,324,$$ which implies that the Poincare series of the OS-algebra is $$P_{OS}(Z)=1+13Z+78Z^2+273Z^3+585Z^4+702Z^5+324Z^6.$$ Hence  $P_{13}(Z)=p(M,Z)=(-1)^{r(M)}P_{OS}(\frac{-1}{Z})=
    Z^6 - 13 Z^5 + 78 Z^4 - 273 Z^3 + 585 Z^2 - 702 Z + 324.$
    Finding the weight polynomials $P_w(Z)$ for $w \le 12$ is straightforward, looking at the relevant subdiagrams 
    of the diagram of the lattice of flats for $M$, and are found in the same way..
    
    Passing to the methods of NBC-sets, one first observes  that the rank of a set of points in $E=\mathbb{P}^2$ is $6-p$, for $p$ the dimension of the space of quadric polynomials that pass through the points of the set. After some arguments using this, one finds that the circuits of $M$ are $13$ sets of cardinality $4$ (the $13$ lines in $\mathbb{P}^2$,  all of them are mapped to conics by the Veronese map, and therefore only span projective planes, not $3$-spaces in $\mathbb{P}^5)$, and furthermore  the $78={\binom{13}{2}}$ line pairs in
    $\mathbb{P}^2$, each of them punctured at the point where the lines in each pair meet, and in addition some sets of cardinality $7$. 

    We then see that there is $1$ NBC-set of cardinality $0$, and $13$ such sets of cardinality $1$, and ${\binom{13}{2}}=78$ such sets of cardinality $2$, and ${\binom{13}{2}}-13=273$ NBC-sets of cardinality $3$. A mechanical search only using the circuits of order $4$ and $6$, explicitly described above,  then gives $585$ NBC-sets of cardinality $4$, and 702 NBC-sets of cardinality $5$. 
    Hence we conclude once again that $$P_{OS}(Z)=1+13Z+78Z^2+273Z^3+585Z^4+702Z^5+324Z^6$$ and that  $P_{13}(Z)=p(M,Z)=(-1)^{r(M)}P_{\mathcal{A}}(\frac{-1}{Z})=
    Z^6 - 13 Z^5 + 78 Z^4 - 273 Z^3 + 585 Z^2 - 702 Z + 324.$
    The last coefficient $324$ was found here as the alternating sum $702-585+273-78+13-1, $ using Proposition \ref{factor}.

    Finding the weight polynomials $P_w(Z)$ for most $w \le 12$ is a little more intricate with the method of NBC-sets, but finding  for example $P_6(Z)$ is easy. This polynomial arises from the $78$ submaximal flats $F$ of $M$.
    These are lines pairs and have cardinality $7$.
    Then $M^* \setminus F$ is a circuit, and hence it is the uniform matroid $U(5,6).$
    Hence its dual is $M/F=U(1,6).$ The circuits of $M/F$ are then all the $2$-point sets. The broken circuits are then all singletons, except $\{1\}.$ The NBC-sets are then $\emptyset$ and $\{1\}.$
    Hence $P_{OS(Z)}=78(Z+1)$, and $P_6(Z)=78Z-78.$

The generalized weight poynomials for $PRM_q(2,2)$ for $q \ge 4$ can be found in \cite{JohnsenPratiharVerdureVeronese}.

\section{Discussion}

The primary objectives of this article have been to place the generalized weight polynomials of matroids within a broader topological framework, and providing tools for coding theorists when the matroid arises from the generator matrices of a linear code. Traditionally, extracting the complete family of generalized weight polynomials has relied heavily on the Tutte polynomial, a method that demands rank and nullity data for every subset of the matroid's ground set. The results presented here establish that such exhaustive enumeration is not strictly necessary. Instead, generalized weight polynomials can be fully recovered using only the inclusion relations and cardinalities of the lattice of flats (and when convenient, via the determination of cyclic flats and associated functions).

Furthermore, this work introduces a practical computational alternative utilizing the Orlik–Solomon (OS) algebra of the lattice of flats. Because the graded dimensions of the OS algebra completely determine the characteristic polynomial , and these dimensions can be computed simply by counting the No Broken Circuit (NBC) sets of each cardinality, coding theorists gain a direct, combinatorial procedure for finding the generalized weight polynomials, and thereafter extracting generalized weight spectra. For the highest weight polynomial one does not have to find the flats, it is enough to look at the circuits. But for this method to be contructive for other weight polynomials than the highest one, one needs to be able to use it for all matroids that are contractions of  the original matroid at its flats,
and then one  must at least determine the flats..

This framework also serves as a alternative to the method of matroid elongations. The elongation approach derives weight polynomials using univariate Möbius numbers with an additional identity as Betti numbers across a sequence of elongation matroids. Conversely, the Formula (\ref{eq:Mobius-flats}) developed in this work depend exclusively on the bivariate Möbius numbers found within the geometric lattice of a single matroid. While all perspectives shed light on different structural properties, the OS algebraic approach is uniquely adapted for topological and Whitney homology interpretations.

When the associated matroid is representable over the complex numbers, this algebraic approach acquires an elegant topological significance. By the fundamental theorem of Orlik and Solomon, the cohomology ring of the complement of a complex hyperplane arrangement is isomorphic to the OS algebra of its underlying matroid. As a result, the local Whitney numbers precisely match the dimensions of the cohomology groups of this arrangement complement. Viewed through this lens, the highest generalized weight polynomial emerges naturally as the Hilbert–Poincaré dual of the Orlik–Solomon algebra, bridging the generalized weight distribution of a linear code directly to geometric topology.






\end{document}